\documentclass[11pt]{amsart}
\usepackage{hyperref}
\hypersetup{linktocpage = true, colorlinks = true, linkcolor = blue, citecolor= red, urlcolor = green}

\usepackage{url}
\usepackage{graphicx, color, bm}
\usepackage{array,float}
\usepackage{amsmath,amsthm}
\usepackage{amssymb, amsfonts, verbatim, subfigure}

\usepackage{mathpazo}
\usepackage{mathrsfs}
\usepackage[mathscr,mathcal]{euscript}

\newcommand{\pd}{{\partial}}

\newcommand{\wt}{\widetilde}
\newcommand{\id}{\text{id}}

\newcommand{\be}{\begin{equation}}
	\newcommand{\ee}{\end{equation}}
\newcommand{\beqn}{\begin{equation}}
	\newcommand{\eeqn}{\end{equation}}
\newcommand{\bp}{\begin{pmatrix}}
	\newcommand{\ep}{\end{pmatrix}}
\newcommand{\bsp}{\left(\begin{smallmatrix}}
	\newcommand{\esp}{\end{smallmatrix}\right)}

\newcommand*\diff{\mathop{}\!\mathrm{d}}

\newcommand{\R}{{\mathbf R}}

\newcommand{\C}{{\mathbf C}}
\newcommand{\Z}{{\mathbf Z}}

\newcommand{\CA}{{\mathcal A}}

\newcommand{\CK}{{\mathcal K}}
\newcommand{\CM}{{\mathcal M}}
\newcommand{\CN}{{\mathcal N}}
\newcommand{\CO}{{\mathcal O}}

\newcommand{\CV}{{\mathcal V}}

\newcommand{\cN}{\CN}

\newcommand{\fsl}{\mathfrak{sl}}
\newcommand{\fgl}{\mathfrak{gl}}
\newcommand{\fg}{\mathfrak{g}}

\newcommand{\norm}[1]{{{:\!{#1}\!:}}}

\DeclareMathOperator{\SVir}{SVir}
\newcommand{\rav}{\mathbb{R}\text{av}}

\numberwithin{equation}{section}
\numberwithin{figure}{section}
\numberwithin{table}{section}

\DeclareMathOperator{\End}{End}
\DeclareMathOperator{\Sym}{Sym}

\DeclareMathOperator{\Spec}{Spec}

\DeclareMathOperator{\Tr}{Tr}

\newcommand{\lie}{\mathfrak}
\newcommand{\op}{\operatorname}

\usepackage{todonotes}

\newtheoremstyle{thm}% name
{7pt}%      Space above
{7pt}%      Space below
{\itshape}%         Body font
{}%         Indent amount (empty = no indent, \parindent = para indent)
{\bf}% Thm head font
{.}%        Punctuation after thm head
{5pt}%     Space after thm head: " " = normal interword space;
{\thmnumber{#2 }\thmname{#1}\thmnote{ (#3)}}%         Thm head spec (can be left empty, meaning `normal')

\newtheoremstyle{def}% name
{7pt}%      Space above
{10pt}%      Space below
{\itshape}%         Body font
{}%         Indent amount (empty = no indent, \parindent = para indent)
{\bf}% Thm head font
{.}%        Punctuation after thm head
{5pt}%     Space after thm head: " " = normal interword space;
{\thmnumber{#2} \thmname{#1}\thmnote{ (#3)}}%         Thm head spec (can be left empty, meaning `normal')

\newtheoremstyle{rem}% name
{4pt}%      Space above
{10pt}%      Space below
{}%         Body font
{}%         Indent amount (empty = no indent, \parindent = para indent)
{\itshape}% Thm head font
{:}%        Punctuation after thm head
{3pt}%     Space after thm head: " " = normal interword space;
{}%         Thm head spec (can be left empty, meaning `normal')

\newtheoremstyle{texttheorem}% name
{8pt}%      Space above
{8pt}%      Space below
{\itshape}%         Body font
{}%         Indent amount (empty = no indent, \parindent = para indent)
{\bf}% Thm head font
{. \hspace{5pt}}%        Punctuation after thm head
{3pt}%     Space after thm head: " " = normal interword space;
{}%         Thm head spec (can be left empty, meaning `normal')

\theoremstyle{thm}

\newtheorem*{theorem*}{Theorem}
\newtheorem*{lemma*}{Lemma}
\newtheorem*{corollary*}{Corollary}
\newtheorem*{proposition*}{Proposition}
\newtheorem*{definition*}{Definition}

\newtheorem{theorem}{Theorem}[subsection]
\newtheorem{thm-def}{Theorem/Definition}[theorem]
\newtheorem{proposition}[theorem]{Proposition}

\newtheorem*{question*}{Question}
\newtheorem{lemma}[theorem]{Lemma}

\newtheorem{corollary}[theorem]{Corollary}

\numberwithin{equation}{subsection}

\theoremstyle{def}

\theoremstyle{rem}
\usepackage{stmaryrd}
\date{}

\title{Higgs and Coulomb branches from Superconformal Raviolo Vertex Algebras}
\author{Niklas Garner}
\address{University of Washington, Seattle}
\email{nkgarner@uw.edu}

\author{Surya Raghavendran}
\address{Yale University}
\email{surya.raghavendran@yale.edu}

\author{Brian R. Williams}
\address{Boston University}
\email{bwill22@bu.edu}

\begin{document}

\begin{abstract}
	We propose a method for extracting the Higgs and Coulomb branches of a three-dimensional $\mathcal{N}=4$ quantum field theory from the algebra of local operators in its holomorphic-topological twist using the formalism of raviolo vertex algebras. Our construction parallels that of the chiral ring and twisted chiral ring of an $N=2$ superconformal vertex operator algebra.
\end{abstract}

\maketitle
\tableofcontents

Higgs and Coulomb branches of vacua in three-dimensional $\CN=4$ quantum field theories have been objects of intense physical and mathematical interest. The Higgs branch $\CM_H$ of an ordinary three-dimensional $\CN=4$ theory of gauged hypermultiplets receives no quantum corrections and the classical answer suffices, cf. \cite{APS}: it can be identified with a hyperk\"{a}hler reduction of the hypermultiplet target space \cite{HKLR}. The Coulomb branch $\CM_C$ of these theories, on the other hand, is typically quite difficult to analyze due to both perturbative and non-perturbative corrections. Work of Bravermann-Finkelberg-Nakajima \cite{Nak, BFN}, inspired by many earlier physical analyses, provided the first mathematically precise definition of the Coulomb branch for a large class of these gauge theories in terms of the equivariant cohomology of a certain moduli space built from the affine Grassmannian for (the complexification of) the gauge group. Alternative constructions in terms of Hilbert schemes of hypertoric varieties, along with investigations of the hyperk\"{a}hler metrics on these spaces, have since appeared \cite{bielawski2023hypertoric}.

Not unlike 2d mirror symmetry, there is a physical duality known as 3d mirror symmetry \cite{IS} that relates two three-dimensional $\CN=4$ theories in a fashion that exchanges these two branches of vacua. This physical duality has far-reaching mathematical connections to symplectic duality \cite{BPW, BLPW}, cf. \cite{BDGH}, and the geometric Langlands correspondence \cite{GaiottoWitten, CreutzigGaiotto}, to name a few. See e.g. the recent expository article \cite{WebsterYoo} for a sampling of these mathematical connections.

These moduli spaces can be extracted from certain topological twists of the corresponding three-dimensional $\CN=4$ theory. More precisely, the algebras of local operators in these twisted theories describe the rings of holomorphic functions (in a choice of complex structure dictated by the twist) on these hyperk\"{a}hler moduli spaces; the Poisson structure arises physically as a secondary product \cite{descent}. The topological $A$ twist, a dimensional reduction of the 4d Donaldson-Witten twist \cite{WittenTQFT}, reproduces $\C[\CM_C]$. Similarly, the topological $B$ twist, also called the Rozansky-Witten twist \cite{RW} but simulteneously studied by Blau-Thompson in the case of pure gauge theory \cite{BT}, reproduces $\C[\CM_H]$. 

It was proposed by Costello-Gaiotto \cite{CostelloGaiotto} that many aspects of these twisted theories, including the Higgs and Coulomb branches \cite{CostelloCreutzigGaiotto}, could be extracted from certain vertex operator algebras of boundary operators. These play a role analogous to that of chiral Wess-Zumino-Witten theory in Chern-Simons theory \cite{WittenJones}.

Many aspects of the topological $A$ and $B$ twists of a three-dimensional $\CN=4$ theory can be accessed more easily by first passing to an intermediate holomorphic-topological ($HT$) twist available to any three-dimensional $\CN\geq2$ theory, see e.g. \cite{ElliottYoo, ESWtax, Butson2, twistedN=4}. Of particular importance for us is that the algebra of local operators in the $HT$ twist contains information about both Higgs and Coulomb branches. Local operators in such an $HT$-twisted theory have the structure of a 1-shifted Poisson vertex algebra \cite{OhYagi}; see also \cite{CostelloDimofteGaiotto, Zeng}. In \cite{GarnerWilliams} an equivalent model for these local operators was developed in terms of the notion of a \textit{raviolo vertex algebra}.
This reformulation offers many parallels with the theory of ordinary vertex algebras and allows for manipulations of the algebras of local operators in three-dimensional holomorphic-topological QFT in much the same way as holomorphic field theories in two dimensions.

A portion of the $\CN=4$ supersymmetry remains after taking the $HT$ twist and, among many other things, there are nilpotent symmetries that can be used to deform the $HT$ twist to the fully topological $A$ and $B$ twists. At the level of local operators, the deformation to the $A$ or $B$ twist is realized by introducing a differential, or perhaps deforming the existing one. We argue in our companion paper \cite{HTenhance} that the $HT$ twists of a wide class of Lagrangian%
\footnote{The symmetry enhancements we find should apply more generally; we restrict our attention to Lagrangian theories to make the discussion explicit.} %
three-dimensional $\CN \geq 2$ supersymmetric field theories actually exhibit an enhancement of their global symmetries, including these residual supersymmetries.

Explicitly, much of the $\lie{osp}(\CN|4)$ symmetry of a three-dimensional superconformal%
\footnote{We find that the $HT$ twist exhibits the following symmetry enhancements regardless of whether the ultraviolet theory is superconformal. Roughly speaking, the theories we consider are expected to flow in the IR to an superconformal theory; as the $HT$-twisted theory is an RG invariant (up to quasi-isomorphism, cf. Section 3.7 of \cite{CostelloDimofteGaiotto}), it should exhibit the symmetry enhancements of this IR SCFT.} %
field theory is not compatible with the $HT$-twist and yet other symmetries become redundant.
It is the cohomology of $\lie{osp}(\CN|4)$ (with respect to the twisting supercharge $Q_{HT}$) that naturally acts on the $HT$ twist of such a theory.
This cohomology can be identified with $\lie{osp}(N|2)$ with $N = \CN-2$, so the $HT$ twist of such a theories admits a natural action by this Lie superalgebra.
For $\cN=2$, the $HT$-twisted symmetry is $\lie{sp}(2) \cong \lie{sl}(2)$, identified with the algebra generated the vector fields $\pd_z$, $z \pd_z$, and $z^2 \pd_z$, and we find this enhances to all holomorphic vector fields $\op{Vect}^{hol}(\C)$ in the examples we consider in \cite{HTenhance}.
For $\cN=4$, the symmetry is $\lie{osp}(2|2)$ and we find an enhancement to the positive part of the $N=2$ Virasoro algebra.

More generally, we find that for a family of supersymmetric Chern--Simons theories with $\cN = 2,3,4,5$, and $6$ supersymmetry this symmetry of their $HT$ twists enhances to the positive part of the $N=0,1,2,3$ and big%
\footnote{The big $N=4$ Virasoro algebra should not be confused with the so-called large $N=4$ algebra. The big $N=4$ algebra arises from the Lie algebra $K_4' = [K_4, K_4]$ in Kac's classification of simple superconformal algebras, see e.g. \cite{KacSCalgs}, whereas the latter is a VOA that is not induced from a Lie algebra, but does arise from Drinfeld-Sokolov reduction of an affine $D(2,1;\alpha)$ algebra \cite{KacWakimoto}. The large $N=4$ algebra can be embedded into the (VOA induced by the) big $N=4$ algebra \cite{GS}.} %
$N=4$ Virasoro algebras, respectively.
In our examples with $\CN=8$ supersymmetry, we find an enhancement from $\lie{osp}(6|2)$ to the exceptional Lie superalgebra $E(1|6)$, which can be identified as the positive part of the exceptional superconformal algebra $CK_6$ discovered by Cheng and Kac \cite{CK}.
The original evidence that $E(1|6)$ should appear as an enhanced symmetry of the $HT$ twist of a 3d superconformal theory comes from its appearance as an enhanced symmetry in a twist of the $AdS_4$ geometry obtained by backreacting $M2$ branes \cite{twistedgraviton}.

Importantly, the raviolo vertex algebras of local operators in the $HT$ twist of a theory with $\CN \geq 2$ supersymmetry should exhibit the above symmetry enhancement: such a raviolo vertex algebra should contain the raviolo analogs of the above sypersymmetric extensions of the Virasoro algebra. There is a natural classification problem for raviolo superconformal algebras analogous to the one initiated by Ramond and Schwarz \cite{RS} for ordinary superconformal algebras.
We will not attempt to provide such a classification here; see e.g. Section 5.10 of \cite{Kac} for a solution to the ordinary classification problem.

In this paper, we exploit one of the parallels between raviolo vertex algebras and ordinary vertex algebras. We show that there are (at least) two natural graded commutative algebras associated to any $N=2$ superconformal raviolo vertex algebra $\CV$ satifying a certain ``BPS bound.'' (See Section \ref{sec:HCprimary} for more details about this BPS bound.) These algebras arise in much the same way as the chiral ring construction in 2d $\CN=(2,2)$ superconformal field theory (SCFT), cf. \cite{LVWchiral}. We call the resulting algebras $A_H[\CV]$ and $A_C[\CV]$ the Higgs and Coulomb branch chiral rings of $\CV$. In fact, as we show in Theorem \ref{thm:poissonHB}, these algebras possess yet more structure: they each have a natural 2-shifted Poisson bracket. This 2-shifted Poisson bracket is the raviolo analog of the Gerstenhaber bracket arising in the chiral ring; see e.g. \cite{WittenZwiebach, PenkavaSchwarz, WuZhu, LianZuckerman}. From the perspective of deforming an $HT$-twisted $\CN=4$ theory to its toplogical $A$ and $B$ twists, this bracket arises via the mechanism of topological descent described in Section 2.4 of \cite{twistedN=4}.

We view the 2-shifted Poisson algebras $A_H[\CV]$, $A_C[\CV]$ as the rings of functions on affine 2-shifted Poisson schemes $\CM_H[\CV]$ and $\CM_C[\CV]$ that we call the Higgs and Coulomb branches of $\CV$. When the $N=2$ superconformal raviolo vertex algebra $\CV$ models local operators in the $HT$ twist of a three-dimensional $\CN=4$ QFT, we expect that the Higgs and Coulomb branches defined by $\CV$ precisely match those of the underlying three-dimensional $\CN=4$ theory. We verify this expectation for two important examples: a free hypermultiplet, whose Higgs branch is identified with $\C^2[-1]$ and whose Coulomb branch is a point, and a free vector multiplet, whose Higgs branch is a (fat) point and whose Coulomb branch is $T^*[2]\C^\times$.

The remainder of the paper is organized as follows. Section \ref{sec:RVAs} is a review of the raviolo vertex algebras introduced in \cite{GarnerWilliams}. This section contains no new results and is meant to orient the reader.

In Section \ref{sec:superconf} we move to $N=2$ superconformal raviolo vertex algebras. In Section \ref{sec:ravSCA} we describe the underlying symmetry algebra, i.e. the raviolo $N=2$ superconformal algebra. We introduce the notion of Higgs and Coulomb branch primary operators in Section \ref{sec:HCprimary}, which are the analogs of chiral and twisted chiral primary operators in 2d $\CN=(2,2)$ SCFT, as well as the BPS bound. We then describe Higgs (and Coulomb) branch flavor symmetries in Section \ref{sec:HCsymmetries}, which lead to Hamiltonian symmetries of the Higgs and Coulomb branches (Corollary \ref{cor:flavorsymmerty}). In Section \ref{sec:SCexamples} we introduce a series of examples, including a free $\CN=4$ hypermultiplet and a free $\CN=4$ vector multiplet. We also present a somewhat unexpected example: the local operators in (the holomorphic-topological twist of) a free chiral multiplet (of $R$-charge $\tfrac{1}{2}$) has an $N=2$ superconformal structure. Finally, in Section \ref{sec:moresusy} we describe several examples of raviolo superconformal algebras with different amounts of supersymmetry arising in the examples of \cite{HTenhance}.

We introduce Higgs and Coulomb branches for any $N=2$ superconformal raviolo vertex algebra satisfying the BPS bound in Section \ref{sec:HC}. In Section \ref{sec:chiralrings} we prove Theorem \ref{thm:poissonHB} showing they are naturally 2-shifted affine Poisson schemes. We introduce the notion of an $N=2$ superconformal raviolo vertex algebra with superpotential in Section \ref{sec:superconfsuperpot}, specializing the construction in Section 4.5 of \cite{GarnerWilliams} to the $N=2$ superconformal setting. Of note is the example of perturbative $\CN=4$ gauge theory described in Section \ref{sec:pertgaugetheory}, where we construct from an $N=2$ superconformal raviolo vertex algebra (possibly with superpotential) having a $\fg$ Higgs branch flavor symmetry a new $N=2$ superconformal raviolo vertex algebra with superpotential whose Higgs branch is the (derived) algebraic symplectic reduction $\CM_H[\CV]/\!\!\!/\fg$ (Theorem \ref{thm:HBpertreduction}). Finally, in Section \ref{sec:toptwisting} we describe a second way to extract two 2-shifted Poisson algebras from any $N=2$ superconformal raviolo vertex algebra regardless of the BPS bound by a cohomological procedure reminiscent of topological twisting, culminating with Theorem \ref{thm:twistingHB}. We show that when the BPS bound is satisfied, there is always a natural map out of the Higgs/Coulomb branch chiral rings (Proposition \ref{prop:naturalmap}) and, moreover, that this map is an isomorphism of 2-shifted Poisson algebras for the examples in Section \ref{sec:SCexamples}.

\subsection*{Acknowledgements}
We would like to thank Zhengping Gui and Natalie M. Paquette for useful conversations during the development of these ideas. 
NG is supported by funds from the Department of Physics and the College of Arts \& Sciences at the University of Washington, Seattle.
SR was supported by Yale University and the Perimeter Institute for Theoretical Physics. Research at Perimeter Institute is supported in part by the Government of Canada, through the Department of Innovation, Science and Economic Development Canada, and by the Province of Ontario, through the Ministry of Colleges and Universities.

\section{Aspects of Raviolo Vertex Algebras}
\label{sec:RVAs}

In this section we present the definition of a raviolo vertex algebra as well as some elementary properties thereof. Our discussion will be expository in nature and this section contains no results beyond those established in~\cite{GarnerWilliams}.

\subsection{The raviolo}
\label{sec:rav}
We start by describing a bit of the geometry underlying raviolo vertex algebras. See Section 1 of \cite{GarnerWilliams} for more details.

An ordinary vertex algebra is built from the geometric space known as the formal punctured disk $D^\times =\Spec \C(\!(z)\!)$, thought of as an infinitesimal, algebraic version of the configuration space of two points on $\C$ (modulo overall translations); somewhat more precisely, the ring $\C(\!(z)\!)$ models holomorphic functions on this configuration space (modulo translations), i.e. the cohomology of the structure sheaf on the formal punctured disk. For example, the notion of a vertex algebraic field on a vector space $V$ is as a certain $\End(V)$-valued formal distribution on the punctured disk.

The geometric space underlying a raviolo vertex algebra is the \emph{formal raviolo} $\rav$. To construct $\rav$, we start with the formal disk $D = \Spec \C[\![z]\!]$. There is a natural inclusion $D^\times \hookrightarrow D$ coming from viewing any formal Taylor series as a formal Laurent series. The formal raviolo is then defined as the scheme-theoretic pushout of two formal disks over a shared formal punctured disk:
\be
	\rav := D \cup_{D^\times} D
\ee
We note that, although $\rav$ is not affine, it can be realized as the spectrum of a commutative dg algebra $\CA$ described in Section 1 of \cite{GarnerWilliams}. In the same way that the formal punctured disk is an infinitesimal, algebraic avatar of the configuration space of two points on $\C$, the formal raviolo is an infinitesimal, algebraic model for configuration space of two points on $\C \times \R$ (modulo overall translations). The replacement for the sheaf of holomorphic functions in this context is the sheaf of functions constant along the leaves of a transverse holomorphic foliation; the cohomology of $\CA$ models the cohomology of this mixed holomorphic-topological analog of the structure sheaf. See Section 1.1 of loc. cit. and references therein for more details. Importantly, this cohomology has support outside of degree $0$. As noted in loc. cit., there is a lot of information lost when passing to cohomology and there should be a natural homotopical enrichment of raviolo vertex algebras coming from working at the level of chains.

The algebraic model for the cohomology of the structure sheaf of $\rav$ used by \cite{GarnerWilliams} is as follows. The degree zero cohomology is isomorphic to holomorphic functions in one variable $\C[z]$. The degree 1 component can be identified with the $\C[z]$-module $\C\langle\Omega^0, \Omega^1, \dots\rangle$, where the action of $z^n$ is given by
\be
z^n \Omega^m = \begin{cases}
	0 & n > m\\
	\Omega^{m-n} & n \leq m
\end{cases}
\ee
and there are no higher cohomology groups. The algebra structure on this ring of functions is as follows: the degree 0 part has its natural ring structure, the multiplication of two degree 1 elements vanishes, and the product of a degree 0 and degree 1 element is given by the above action together with $z^m \Omega^n = \Omega^n z$.
We will call this the ring of polynomial functions on the raviolo and denote it $\CK_{poly}$; this is the raviolo analog of Laurent polynomials $\C[z,z^{-1}]$, where $\Omega^m$ plays the role of $z^{-m-1}$. We denote the completion of this ring that allows for formal Taylor series in $z$ by $\CK$; this is the raviolo analog of formal Laurent series $\C(\!(z)\!)$, we correspondingly call elements of $\CK$ formal raviolo Laurent series. We will also need the completion $\CK_{dist}$ that allows for infinite sums in both degrees; this is the analog of $\C[\![z,z^{-1}]\!]$ and we call elements of $\CK_{dist}$ formal raviolo distributions. Note that $\CK_{dist}$ does not have the structure of an algebra, but it is a $\CK$-module. There is a natural endomorphism of all of these spaces, denoted $\pd_z$, that acts on the generators as
\be
	\pd_z z^n = n z^{n-1} \qquad \pd_z \Omega^m = -(m+1) \Omega^{m+1}
\ee
This is a derivation of the algebras $\CK_{poly}$ and $\CK$. Finally, we note that there is a natural linear map $\diff z:\CK_{poly} \to \C[-1]$ defined by
\be
	\oint_0 \diff z z^n = 0 \qquad \oint_0 \diff z \Omega^n = \delta_{n,0}
\ee
for any $n \geq 0$, extended by linearity. This is the raviolo analog of the residue pairing in complex analysis and serves as an algebraic avatar of integrating over a small 2-sphere centered at $0$, whence the notation. This residue map is inherited by the completions $\CK$ and $\CK_{dist}$.

To simplify notation, for any $\C$-algebra $R$ we denote by $R\langle\!\langle z\rangle\!\rangle$ the ring for $R$-valued formal raviolo Laurent series, e.g. $\CK = \C\langle\!\langle z\rangle\!\rangle$. We will use a superscript to differentiate different copies of $\CK$, e.g. $\CK^z = \C\langle\!\langle z\rangle\!\rangle$ versus $\CK^w = \C\langle\!\langle w\rangle\!\rangle$, and differentiate the corresponding degree $1$ generators with a subscript, e.g. $\Omega^m_z$ versus $\Omega^m_w$.

We can similarly consider multivariate versions of the spaces of raviolo functions and distributions. We denote by $\CK^{z,w}_{dist}$ the space of bivariate formal ravioli distributions. This space is concentrated in degrees $0$, $1$, and $2$; degree $0$ is identified with bivariant formal Taylor series $\C[\![z,w]\!]$; degree $1$ is identified with formal series of elements of the form $z^n \Omega^m_w$ and $w^m \Omega^n$; and degree $2$ is identified with formal series elements of the form $\Omega^n_z \Omega^m_w = - \Omega^m_w \Omega^n_z$. As with formal series, there are two natural inclusions
\be
	\C\langle\!\langle z\rangle\!\rangle\langle\!\langle w\rangle\!\rangle \hookrightarrow \CK_{dist}^{z,w} \hookleftarrow \C\langle\!\langle w\rangle\!\rangle\langle\!\langle z\rangle\!\rangle
\ee
serving as the raviolo analog of the inclusions of $\C(\!(z)\!)(\!(w)\!)$ and $\C(\!(w)\!)(\!(z)\!)$ into $\C[\![z^{\pm1}, w^{\pm1}]\!]$.

The last ring we will need is the raviolo analog of $\C[\![z,w]\!][z^{-1}, w^{-1}, (z-w)^{-1}]$, denoted $\CK_{z,w,z-w}$ in \cite{GarnerWilliams}. This serves as an algebraic avatar of the ring of functions on the configuration space of two points. The degree zero part of $\CK_{z,w,z-w}$ is simply bivariate Taylor series $\C[\![z,w]\!]$. As a module for the degree zero part, $\CK_{z,w,z-w}$  is generated by three towers of degree $1$ elements $\Omega^m_x$, for $x = z$, $w$, $z-w$, with the expected action by degree 0 elements, e.g. $x^n \Omega^m_x = \Omega^{m-n}_x$ if $n \leq m$ and $0$ otherwise. There is also a degree $2$ relation
\be
	\Omega^0_{z-w} \Omega^0_z + \Omega^0_w \Omega^0_{z-w} + \Omega^0_z \Omega^0_w = 0
\ee
together with all relations obtained from this by acting with $z, \pd_z$ and $w,\pd_w$. The fundamental property of $\CK_{z,w,z-w}$ is captured by the following lemma.

\begin{lemma}[Lemma 1.4.1 of \cite{GarnerWilliams}]
\label{lem:expand}
	There are graded algebra maps from $\CK_{z,w,z-w}$ to $\C\langle\!\langle z\rangle\!\rangle\langle\!\langle w\rangle\!\rangle$, $\C\langle\!\langle w\rangle\!\rangle\langle\!\langle z\rangle\!\rangle$, and $\C\langle\!\langle w\rangle\!\rangle\langle\!\langle z-w\rangle\!\rangle$.
\end{lemma}

The graded algebra map $\CK_{z,w,z-w} \to \C\langle\!\langle z\rangle\!\rangle\langle\!\langle w\rangle\!\rangle$ arises from Taylor expanding $\Omega^m_{z-w}$ ``for small $w$'' as
\be
	\Omega^m_{z-w} \mapsto \sum_{n \geq 0} \binom{m+n}{n} w^n \Omega^{n+m}_z
\ee
Similarly, the map $\CK_{z,w,z-w} \to \C\langle\!\langle w\rangle\!\rangle\langle\!\langle z\rangle\!\rangle$ arises from expanding $\Omega^m_{z-w}$ ``for small $z$'' as
\be
	\Omega^m_{z-w} \mapsto (-1)^m\sum_{n \geq 0} \binom{m+n}{n} z^n \Omega^{n+m}_w
\ee
and the map $\CK_{z,w,z-w} \to \C\langle\!\langle w\rangle\!\rangle\langle\!\langle z-w\rangle\!\rangle$ arises from expanding $\Omega^m_z$ ``for small $z-w$'' as
\be
	\Omega^m_{z} \mapsto \sum_{n \geq 0}(-1)^n \binom{m+n}{n} (z-w)^n \Omega^{n+m}_w
\ee

\subsection{Raviolo vertex algebras}
We now move to the structure of a \emph{raviolo vertex algebra}. We will be rather brief; see Section 2 of \cite{GarnerWilliams} for a more detailed treatment.

We start with a graded vector space $\CV$, denoting the subspace of homogeneous elements of degree $r$ by $\CV^r$. Physically, $\CV$ is the vector space of local operators at, say, $0 \in \C \times \R$ in a three-dimensional holomorphic-topological QFT with $r$ being a cohomological grading. We will call this grading \emph{$R$-charge}. We will be interested in a certain class of $\End(\CV)$-valued formal raviolo distributions; we will write such a formal distribution as
\be
	A(z) = \sum_{m < 0} z^{-m-1} A_m + \sum_{m \geq 0} \Omega^m_z A_m
\ee
We say that $A(z)$ is a \emph{raviolo field} or simply a \emph{field} if for any $v \in \CV$ we have $A_m v = 0$ for $m \gg 0$. We say that $A(z)$ is \emph{homogeneous} of degree $|A|$ if $A_m$ is an endomorphism of degree $|A|$ for $m < 0$ and degree $|A|-1$ for degree $m \geq 0$.%
\footnote{The reason for this shift is due to the fact that $\Omega^m_z$ is degree $1$.} %

The vector space $\CV$ of local operators is equipped with a natural degree 0 \emph{translation operator} $\pd: \CV^{r} \to \CV^{r}$, and has a distinguished \emph{vacuum vector} $|0\rangle \in \CV^0$ that is annihilated by $\pd$: $\pd 1 = \pd |0\rangle = 0$. Physically, $\pd$ sends a local operator at $0$ to its holomorphic derivative and the vector $|0\rangle$ corresponding to the trivial/identity local operator $1 \leftrightarrow |0\rangle$.

The action of a given local operator $O \in \CV$ on $\CV$ is encoded in a field $Y(O,z)$ on $\CV$, which we write as
\be
	Y(O,z) = \sum_{m < 0} z^{-m-1} O_{(m)} + \sum_{m \geq0} \Omega^z_m O_{(m)}
\ee
By an abuse of notation we will usually denote $Y(O,z) = O(z)$; the map $Y(-,z): \CV \to \End(\CV)\otimes \CK_{dist}$ will be called the \emph{state-field} or \emph{state-operator correspondence}. The state-operator correspondence is required to be compatible with the gradings on $\CV$ and $\CK_{dist}$, i.e. the $R$-charges of the modes $O_{(n)}$ in the field $O(z)$ agree with those mentioned above. We can use the raviolo residue pairing to express the endomorphisms $O_{(m)}$ via integrals of $O(z)$: for $m \geq 0$ we have
\be
	O_{(m)} = \oint_0 \diff z\, z^m O(z) \qquad O_{(-m-1)} = \oint_0 \diff z\, \Omega^{m}_z O(z)
\ee 

We require the state-operator correspondence satisfies the \emph{vacuum axiom}
\be
	Y(|0\rangle, z) = \id_\CV \qquad Y(O,z) |0\rangle \in \CV[\![z]\!] \qquad Y(O,0)|0\rangle = O\,.
\ee
as well as the \emph{translation axiom}
\be
	[\pd, Y(O,z)] = \pd_z Y(O,z) = \sum_{n \geq 0} z^n \bigg( (n+1)O_{n+1}\bigg) + \Omega_n \bigg(-n\Pi_{n-1}\bigg)
\ee
and $\pd |0\rangle = 0$ to ensure that the state-operator correspondence intertwines the action of $\pd_z$ on $\CK$ and the action of $\pd$ on $\CV$. Finally, we note that there is a notion of \emph{locality} of the above fields on $\CV$: we say two (homogeneous) fields $A(z), B(w)$ are \emph{mutually local} if for every $v \in \CV$ and $\varphi \in \CV^\vee$ (the linear dual of $\CV$) the matrix elements
\be
	\langle\varphi, A(z) B(w) v\rangle \in \C\langle\!\langle z\rangle\!\rangle\langle\!\langle w\rangle\!\rangle
\ee
and
\be
(-1)^{|A||B|}\langle\varphi, B(w) A(z) v\rangle \in \C\langle\!\langle w\rangle\!\rangle\langle\!\langle z\rangle\!\rangle
\ee
are expansions of the same element of $\CK_{z,w,z-w}$, cf. Lemma \ref{lem:expand}, and moreover the coefficient of $\Omega^l_{z-w}$ vanishes for all $l \gg 0$. We assume the \emph{locality axiom} that the fields $Y(O_1, z)$ and $Y(O_2,w)$ are mutually local for any $O_1$, $O_2$. As with vertex algebraic fields, there are many equivalent formulations of mutual locality of two raviolo fields; see Proposition 2.2.2 of \cite{GarnerWilliams}.

A \emph{raviolo vertex algebra} is the data $(\CV, |0\rangle, \pd, Y)$ of a graded vector space $\CV = \bigoplus_{r} \CV^{r}$ together with a vacuum vector $|0\rangle \in \CV^0$, a translation operator $\pd: \CV^r \to \CV^r$, and state-operator correspondence $Y$ compatible with the grading. This data must satisfy the above vacuum axiom, translation axiom, and locality axiom. The notions of subalgebras, ideals, morphisms, and derivations of raviolo vertex algebras are entirely parallel to those of ordinary vertex algebras; see Section 2.3 of \cite{GarnerWilliams} for more details.

We note that there are variants of raviolo vertex algebras that allow for additional gradings on the vector space $\CV$. For example, we say that $\CV$ has a \emph{spin grading} if it is bigraded, with homogeneous components $\CV^{r,(j)}$, such that $|0\rangle$ has bidegree $(r,j) = (0,0)$, $\pd$ is homogeneous of bidegree $(0,1)$, and the state-operator correspondence $Y$ has bidegree $(0,0)$. This spin grading does not contribute signs when manipulating the raviolo vertex algebra. We say $\CV$ has a \emph{super grading} if it has an additional $\Z/2$ grading such that $|0\rangle$, $\pd$, and $Y$ are all even with respect to this grading. This additional $\Z/2$ grading does contribute to the Koszul rule of signs and amounts to working over $\Z$-graded super vector spaces. When there is a super grading, we denote by $|A|$ the totalized grading (mod 2); we call operators/fields with even totalized grading \emph{bosons} and call them \emph{fermions} if their totalized grading is odd.%
\footnote{We will also consider cases where there are half-integral $R$-charges and so it does not make sense to ask for the totalized grading to be even or odd. When this is the case, we use fermionic/bosonic to describe the signs arising in algebraic manipulations.} %
Finally, we note that it is possible to define a raviolo vertex algebra over a graded commutative $\C$-algebra $S$; see Section 2.5 of \cite{GarnerWilliams} for more details.

\subsection{The operator product expansion}
One of the most important properties satisfied by a raviolo vertex algebra is associativity, cf. Theorem 3.3.1 of \cite{GarnerWilliams}. Given two operators $O_1, O_2$ there is an \emph{operator product expansion} (OPE) of the corresponding fields:
\be
\begin{aligned}
	O_1(z) O_2(w) & = \sum_{m < 0} (z-w)^{-m-1} (O_{1,{(m)}} O_2)(w) + \sum_{m \geq 0}\Omega^m_{z-w} (O_{1,(m)} O_2)(w)\\
	& = \norm{O_1(z) O_2(w)} + \sum_{m \geq 0} \Omega^m_{z-w} (O_{1,(m)} O_2)(w)\\
\end{aligned}
\ee
where $(O_{1,(m)} O_2)(w) = Y(O_{1,(m)}O_2,w)$. We define the \emph{normally-ordered product} as
\be
\begin{aligned}
	\norm{O_1(z)O_2(w)} & = \sum_{n \geq 0} (z-w)^m (O_{1,(m)} O_2)(w)\\
	& = O_1(z)_+ O_2(w) + (-1)^{|O_1||O_2|} O_2(w) O_1(z)
\end{aligned}
\ee
where for any formal raviolo distribution
\be
	f(z) = \sum_{m < 0} z^{-m-1} f_m + \sum_{m \geq 0} \Omega^{m} f_m
\ee
we write
\be
	f(z)_+ = \sum_{m < 0} z^{-m-1} f_m\,, \quad f(z)_- = \sum_{m \geq 0} \Omega^m_z f_m
\ee
As for vertex algebras, above the left-hand side of the OPE should be understood as its expansion in small $z-w$. When $O_1(z)$, $O_2(w)$ are fields, the specialization of $\norm{O_1(z) O_2(w)}$ at $w = z$ is itself a field that we denote $\norm{O_1 O_2}(z)$. We note that $(O_{1,(m)} O_2)(z)$ for $m < 0$ is related to this specialized normally-ordered product as
\be
	(O_{1,(m)} O_2)(z) = \tfrac{1}{(-m-1)!}\norm{(\pd^m O_1) O_2}(z)
\ee

We call the coefficients of $\Omega^m_{z-w}$ the \emph{singular terms} of the OPE and the call the remaining terms \emph{regular}; for brevity, we write the OPE as
\be
	O_1(z) O_2(w) \sim \sum_{m \geq 0} \Omega^m_{z-w} (O_{1,(m)} O_2)(w)
\ee
Our assumption that $O_1(z), O_2(w)$ are fields on $\CV$ ensures that there are at most a finite number of singular terms in the OPE of any two fields. We note that the OPE has the following skew-symmetry property:
\be
	O_2(z) O_1(w) \sim \sum_{m \geq 0} \Omega^m_{z-w} \bigg((-1)^{|O_1||O_2|}\sum_{l\geq0} \frac{(-1)^{m+l}}{l!} \pd^l (O_{1,(m+l)} O_2)(w)\bigg)
\ee
cf. Proposition 3.2.2 of \cite{GarnerWilliams}.

We can extract the fields $(O_{1,(m)} O_2)(w)$ using the raviolo residues. For $m \geq 0$ we have
\be
\begin{aligned}
	(O_{1,(m)}O_2)(w) & = \oint_{w} \diff z (z-w)^m O_1(z) O_2(w)\\
	& := \oint_0 \diff z (z-w)^m O_1(z) O_2(w)\\
	& \quad - (-1)^{|O_1||O_2|} \oint_0 \diff z (z-w)^m O_2(w)O_1(z)
\end{aligned}
\ee
and
\be
\begin{aligned}
	(O_{1,(-m-1)}O_2)(w) & = \oint_{w} \diff z \Omega^m_{z-w} O_1(z) O_2(w)\\
	& := \oint_0 \diff z \Omega^m_{z-w} O_1(z) O_2(w)\\
	& \quad - (-1)^{|O_1||O_2|} \oint_0 \diff z \Omega^m_{z-w} O_2(w)O_1(z)
\end{aligned}
\ee
In this last expression, $\Omega^m_{z-w}$ must be interpreted as its expansion in the appropriate region, i.e. for small $w$ in the first term and for small $z$ in the second.

As described in Section 3.4 of loc. cit., the specialized normally-ordered product $\norm{O_1 O_2}(z)$ (is the field corresponding to) the physical operator product whereas $(O_{1,(m)} O_2)(z)$ (are the fields corresponding to) the holomorphic-topological descent brackets $\{\!\{O_1, O_2\}\!\}^{(n)}$ of \cite{OhYagi}.%
\footnote{More precisely, the shifted $\lambda$-bracket of loc. cit. is a generating function of these brackets.} %

\section{Superconformal Raviolo Vertex Algebras}
\label{sec:superconf}

In this section we focus on raviolo vertex algebras with an $N=2$ superconformal structure. From the perspective of $HT$-twisted three-dimensional $\CN \geq 2$ QFT, such a superconformal structure arises when the underlying theory has $\CN \geq 4$ supersymmetry. As we shall see in Section \ref{sec:toptwisting}, the superconformal structure succinctly encodes how to deform a mixed holomorphic-topological QFT to two fully topological theories: there are nilpotent symmetries coming from the $N=2$ superconformal symmetry that deform to a the topological $A$ and $B$ twists, cf. \cite{twistedN=4, ESWtax}. We briefly consider different amounts of supersymmetry in Section \ref{sec:moresusy}.

\subsection{The raviolo $N=2$ superconformal algebra}
\label{sec:ravSCA}

We start by describing the underlying symmetry algebra, cf. \cite{HTenhance}. We start with the raviolo Virasoro algebra $Vir$. This has bosonic generators $G_{m}$ and fermionic generators $\xi$ and $\Gamma_m$, for $m \geq 0$, with the following Lie brackets:%
\footnote{We use the convention that $[a,b]$ denotes the \emph{graded} commutator, i.e. it is a commutator if one of $a$ or $b$ is bosonic and an anticommutator if both $a$ and $b$ are fermionic.} %
\be\label{eqn:vir}
\begin{aligned}[]
	[G_m, G_n] = (m-n) G_{m+n-1} \hspace{2cm} [\Gamma_m, \Gamma_n] = 0\\
	[G_m, \Gamma_n] = \begin{cases} 
		0 & n+3 < m\\
		\frac{m(m-1)(m-2)}{12}\xi & n+3 = m\\
		0 & n+2 = m\\
		(m+n+1) \Gamma_{n-m+1} & n+1 \geq m\\
	\end{cases} \hspace{0.5cm}
\end{aligned}
\ee
and where $\xi$ is central. We also have an abelian raviolo current algebra (at level $\xi/3$), which has bosonic generators $S_m$ and fermionic generators $\sigma_m$ and brackets
\be\label{eqn:abeliancurrent}
\begin{aligned}[]
	[S_{m}, S_{n}] = 0 \hspace{2cm} [\sigma_{m}, \sigma_{n}] = 0\\
	[S_{m}, \sigma_{n}] = \begin{cases} 
		m \xi/3 & n+1 = m\\
		0 & \text{ else}
	\end{cases} \hspace{0.5cm}
\end{aligned}
\ee
The brackets of the raviolo Virasoro algebra on this abelian current algebra are as follows: the bosonic modes $G_n$ act as
\be
	[G_m, S_{n}] = -n S_{m+n-1} \qquad [G_m,\sigma_n] = \begin{cases}
	0 & n + 1 < m\\
	(n+1)\sigma_{n-m+1} & n +1 \geq m\\
\end{cases}
\ee
and the fermionic modes $\Gamma_m$ act as
\be
[\Gamma_m, S_{n}] = \begin{cases}
	0 & n > m+1\\
	-n\sigma_{m-n+1} & n \leq m+1\\
\end{cases} \qquad [\Gamma_m, \sigma_{n}] = 0
\ee
The bosonic generators $G_m$ and $S_m$ generate the bosonic positive part of the ordinary $N=2$ superconformal algebra.

In addition to these generators, there are additional fermionic generators $\theta^\pm_m$ and bosonic generators $Q^\pm_m$. The action of the Virasoro algebra on these generators is as follows: the bosonic modes $G_n$ act as
\be
\begin{aligned}[]
	[G_m, \theta^\pm_{n}] & = \big(\tfrac{1}{2}m -n\big)\theta^\pm_{m+n-1}\\
	[G_m, Q^\pm_n] & = \begin{cases}
		0 & n + 1 < m\\
		\big(\tfrac{1}{2}m+n+1\big)Q^\pm_{n-m+1} & n+1 \geq m\\
	\end{cases}
\end{aligned}
\ee
and the fermionic modes $\Gamma_m$ act as
\be
\begin{aligned}[]
	[\Gamma_m, \theta^\pm_{n}] & = \begin{cases}
		0 & n > m +1\\
		-\big(\tfrac{1}{2}(m+1)+n\big)Q^\pm_{m-n+1} & n \leq m+1\\
	\end{cases}\\
	[\Gamma_m, Q^\pm_{n}] & = 0
\end{aligned}
\ee
The action of the current algebra is as follows: the bosonic modes $S_{m}$ act as
\be
	[S_{m}, \theta^\pm_{n}] = \pm \theta^\pm_{n} \qquad [S_{m}, Q^\pm_{n}] = \begin{cases}
		0 & n < m\\
		\pm Q^\pm_{n-m} & n \geq m\\
	\end{cases}
\ee
and the fermionic modes $\sigma_{m}$ act as
\be
	[\sigma_{m}, \theta^\pm_{n}] = \begin{cases}
		0 & n > m\\
		\pm Q^\pm_{m-n} & n \leq m\\
	\end{cases} \qquad [\sigma_{m}, Q^\pm_{n}] = 0
\ee
Finally, the brackets of these generators with themselves are given by
\be
	[\theta^\pm_m, \theta^\pm_n] = 0 \hspace{2cm} [\theta^\pm_m, Q^\pm_n] = 0 \hspace{2cm} [Q^\pm_m, Q^\pm_n] = 0
\ee
together with
\be
\begin{aligned}[]
	[\theta^\pm_m, \theta^\mp_n] = G_{m+n} \pm \tfrac{1}{2}(m-n) S_{m+n-1} \qquad [Q^\pm_m, Q^\mp_n] = 0\\
	[\theta^\pm_m, Q^\mp_n] = \begin{cases} 
		0 & n+2 < m\\
		-\frac{m(m-1)}{6}\xi & n+2 = m\\
		\mp m \sigma_{0} & n+1 = m\\
		-\Gamma_{n-m} \mp \tfrac{1}{2}(m+n+1) \sigma_{n-m+1} & n \geq m\\
	\end{cases} \hspace{-0.25cm}
\end{aligned}
\ee
We call this Lie algebra the \emph{raviolo $N=2$ superconformal algebra} and denote it $SVir^{N=2}$.

We can use this Lie algebra to construct a (universal) raviolo $N=2$ superconformal algebra via induction. We consider the positive subalgebra $SVir^{N=2}_+$ generated by the central generator $\xi$, the bosonic generators $G_m$, $S_m$, and the fermionic generators $\theta^\pm_n$. This can be identified with the positive subalgebra of the ordinary $N=2$ superconformal algebra, together with the central generator $\xi$. We then give $\C$ the structure of a $SVir^{N=2}_+$-module by declaring all of these generators act trivially and then consider the $SVir^{N=2}$-module
\be
	\SVir^{N=2} = U SVir^{N=2} \otimes_{USVir^{N=2}_+} \C
\ee
As $\xi$ is central and acts as $0$ on $\C$, it follows that $\xi$ acts as $0$ on all of $\SVir^{N=2}$; the resulting raviolo vertex algebra is analogous to the Virasoro vertex algebra at vanishing central charge. In order to incorporate the level $\xi$, we can instead give $\C[\xi]$ the structure of a $SVir^{N=2}_+$ module, where $\xi$ acts by multiplication by $\xi$ and the remaining generators act as zero. We then consider the $SVir^{N=2}$-module
\be
	\SVir^{N=2}_{univ} = U SVir^{N=2} \otimes_{U SVir^{N=2}_+} \C[\xi]
\ee
We can realize $\SVir^{N=2}$ as the quotient of $\SVir^{N=2}_{univ}$ by vectors proportional to $\xi$. Note that because $\xi$ is fermionic, i.e. $\xi^2 = 0$ in $U SVir^{N=2}$, this is the only maximal ideal of $\C[\xi]$ and hence the central charge must either vanish, leading to $\SVir^{N=2}$, or it must be unconstrained, leading to $\SVir^{N=2}_{univ}$. 

We identify the vacuum vector with the image of $1 \otimes 1$ in the quotient and the translation operator is identified as $\pd = G_0$. The image of the linear state $\Gamma_0 \otimes 1$ corresponds to the fermionic field (of $R$-charge 1 and spin 2)
\be
	\Gamma(z) = \sum_{n \geq 0} z^{n} \Gamma_{n} + \Omega^n_z G_{n}\,,
\ee
and the image of the linear state $\sigma_0 \otimes 1$ corresponds to the fermionic field (of $R$-charge 1 and spin 1)
\be
	\sigma(z) = \sum_{n \geq 0} z^{n} \sigma_{n} + \Omega^n_z S_{n}\,.
\ee
The imagine of the linear states $Q^\pm_0 \otimes 1$ correspond to the bosonic fields (of $R$-charge 1 and spin $\tfrac{3}{2}$)
\be
	Q^\pm(z) = \sum_{n \geq 0} z^{n} Q^\pm_{n} + \Omega^n_z \theta^\pm_{n}\,.
\ee

\begin{proposition}
	$\SVir^{N=2}$ and $\SVir^{N=2}_{univ}$ are raviolo vertex algebras with spin grading.
\end{proposition}

\begin{proof}
	With the above data, the proof is a simple application of the raviolo reconstruction theorem, i.e. Proposition 4.0.1 of \cite{GarnerWilliams}.
\end{proof}

We note that $\SVir^{N=2}_{univ}$ is strongly generated (over $\C[\xi]$) by the fermionic fields $\Gamma, \sigma$ and the bosonic fields $Q^\pm$, i.e. the vector space is spanned (as an $\C[\xi]$-module) by normally-ordered products of these fields acting on the vacuum state $|0\rangle$. The above commutators of the modes of the strong generators translate to the following OPEs. The OPE of the fermion $\Gamma$ with itself is given by
\be
	\Gamma(z) \Gamma(w) \sim \Omega_{z-w}^3(\xi/2) + 2 \Omega_{z-w}^1 \Gamma(w) + \Omega_{z-w}^0 \pd \Gamma(w)
\ee
corresponding to the raviolo Virasoro algebra. The OPE of the fermion $\sigma$ with itself is given by
\be
	\sigma(z) \sigma(w) \sim \Omega^1_{z-w} (\xi/3)
\ee
generating a copy of the $\fgl(1)$ raviolo current algebra. The OPE of these two fermions takes the form
\be
	\Gamma(z) \sigma(w) \sim \Omega^1_{z-w} \sigma(w) + \Omega^0_{z-w} \pd \sigma(w)
\ee
and says that $\sigma(z)$ transforms as a raviolo Virasoro primary of spin 1. The bosonic fields $Q^\pm$ transform as raviolo Virasoro primaries of spin $\frac{3}{2}$
\be
	\Gamma(z) Q^\pm(w) \sim \tfrac{3}{2} \Omega^1_{z-w} Q^\pm(w) + \Omega^0_{z-w} \pd Q^\pm(w)\,,
\ee
and as current algebra primaries of weight $\pm 1$
\be
	\sigma(z) Q^\pm(w) \sim \pm \Omega^0_{z-w} Q^\pm(w)
\ee
Finally, the OPE of $Q_\pm$ with itself is regular and the OPE of $Q^\pm$ with $Q^\mp$ is given by
\be
	Q^\pm(z) Q^{\mp}(w) \sim \Omega_{z-w}^2 (-\xi/3) \mp \Omega_{z-w}^1 \sigma(w) + \Omega_{z-w}^0\big(-\Gamma(w) \mp \tfrac{1}{2}\pd \sigma(w)\big)
\ee
We will call a conformal raviolo vertex algebra (of central charge $\xi$), cf. Definition 4.4.2  of \cite{GarnerWilliams}, with a choice of $\sigma$ and $Q^\pm$ satisfying the above OPEs an \emph{$N=2$ superconformal} raviolo vertex algebra (of central charge $\xi$). All of the explicit examples we consider will have vanishing central charge $\xi = 0$. Any $N=2$ superconformal raviolo vertex algebra admits an additional grading by $S_0$; we refer to $\sigma(z)$ as the \emph{superconformal current} and to weights with respect to $S_0 = \sigma_{(0)}$ as \emph{$S$-charge}.

There is a $\Z_2$ \emph{mirror automorphism} of the $N=2$ superconformal algebra that exchanges the supercharges $Q^+ \leftrightarrow Q^-$ and negates the superconformal current $\sigma \leftrightarrow -\sigma$. We say that two $N=2$ superconformal raviolo vertex algebras are a \emph{three-dimensional mirror pair} if there is an isomorphism between them that intertwines their actions of the $N=2$ superconformal raviolo vertex algebra with this mirror automorphism. We expect that the $N=2$ superconformal raviolo vertex algebras of a mirror pair of three-dimensional $\CN=4$ QFTs are themselves a mirror pair.

The $N=2$ superconformal raviolo vertex algebra is the raviolo analog of the $N=2$ superconformal algebra in the theory of vertex algebras. Many aspects of this section are direct consequences of this simple fact; we largely adapt the arguments of \cite{LVWchiral} to the raviolo vertex algebra setting.

There is a copy of $\mathfrak{osp}(2|2)$ inside the mode algebra of the $N=2$ superconformal raviolo vertex algebra: the mode $\sigma_{(0)}$ of the current $\sigma$ and the modes $\Gamma_{(0)}$, $\Gamma_{(1)}$, $\Gamma_{(2)}$ of $\Gamma$ generate the bosonic subalgebra and the modes $Q^\pm_{(0)}, Q^\pm_{(1)}$ of $Q^\pm$ generate the fermionic subspace. We interpret this copy of $\mathfrak{osp}(2|2)$ as the $HT$-twist, i.e. $Q_{HT}$-cohomology, of the three-dimensional $\CN=4$ superconformal algebra $\mathfrak{osp}(4|4)$, cf. \cite{HTenhance}.

\subsection{Higgs and Coulomb branch primary operators}
\label{sec:HCprimary}
The first class of operators we consider are the raviolo analog of chiral and twisted-chiral primary operators in an $N=2$ superconformal vertex algebra. Suppose $O$ is a raviolo Virasoro primary of spin $j$ and a current algebra primary of $S$-charge $q$, i.e. its OPEs with $\Gamma$ and $\sigma$ are then given by
\be
	\Gamma(z) O(w) \sim j \Omega^1_{z-w} O(w) + \Omega^0_{z-w} \pd O(w) \qquad \sigma(z) O(w) \sim q \Omega^0_{z-w} O(w)
\ee
We say that $O$ is a \emph{Higgs branch primary operator} if its OPEs with $Q^\pm$ take the form
\be
	Q^+(z) O(w) \sim 0 \qquad Q^-(z) O(w) \sim \Omega^0_{z-w} \Psi_O(w)
\ee
where $\Psi_O$, called the \emph{superpartner} of $O$, is a second raviolo Virasoro primary of spin $j+\frac{1}{2}$ and current algebra primary of $S$-charge $q-1$ of the opposite parity; the operators $O$ and $\Psi_O$ have the same $R$-charge. Similarly, we say it is a \emph{Coulomb branch primary operator} if it has the above OPEs with the roles of $Q^+$ and $Q^-$ exchanged. The following lemma is an immediate consequence of superconformal symmetry:

\begin{lemma}
	\label{lem:HBP}
	Let $O$ be a Higgs branch primary operator of spin $j$ and $S$-charge $q$, then $j = \frac{q}{2}$. Moreover, the OPEs of $Q^\pm$ with the superpartner $\Psi_O$ are given by
	\begin{equation*}
		Q^+(z) \Psi_O(w) \sim q\Omega^1_{z-w} O(w) + \Omega^0_{z-w} \pd O(w) \qquad Q^-(z) \Psi_O(w) \sim 0
	\end{equation*}
\end{lemma}

There is a similar statement for Coulomb branch primary operators, where the spin and $S$-charge are related as $j = - \frac{q}{2}$; the OPEs with the superpartner are given by exchanging $Q^+ \leftrightarrow Q^-$ and negating the $S$-charge $q \to -q$.

We now introduce a special class of $N=2$ superconformal raviolo vertex algebras: we say that $\CV$ \emph{satisfies the BPS bound} if
\be
\label{eq:BPSbound}
	\CV^{r,(j),q} = 0\,,\quad  j < \frac{|q|}{2}
\ee
where $\CV^{r,(j),q}$ is the subspace of vectors with spin $j$, $R$-charge $r$, and $S$-charge $q$. Note that the BPS bound implies all operators have non-negative spins $j \geq 0$.

For the $N=2$ superconformal raviolo vertex algebra of local operators in an $HT$-twisted unitary three-dimensional $\CN=4$ SCFT, we expect this bound to be a direct consequence of the usual BPS bound; we will give several examples of $N=2$ superconformal raviolo vertex algebras satifying this bound in Section \ref{sec:SCexamples}.

\begin{lemma}
	\label{lem:BPSboundprimaries}
	Let $\CV$ be an $N=2$ superconformal raviolo vertex algebra that satisfies the BPS bound, then operators saturating the BPS bound are necessarily raviolo Virasoro primaries and superconformal current algebra primaries.
\end{lemma}

\begin{proof}
	Suppose $O$ saturates the BPS bound, i.e. $j = \frac{|q|}{2}$. The operators $S_n O$ and $G_{n+1} O$ for $n > 0$ have the same $S$-charge as $O$ but spin $j - n < j = \frac{|q|}{2}$, hence they must vanish, i.e. $O$ is a raviolo Virasoro primary and a superconformal current algebra primary.
\end{proof}

\begin{corollary}
	\label{cor:HBP2}
	Let $\CV$ be an $N=2$ superconformal raviolo vertex algebra satisfying the BPS bound and let $O\in \CV$ be an operator of spin $j$ and $S$-charge $q$. The following are equivalent:
	\begin{itemize}
		\item[1)] $O$ is a Higgs branch primary operator
		\item[2)] $O$ has $j = \frac{q}{2}$ and $Q^+_{(0)} O = 0$, $Q^-_{(1)} O = 0$
	\end{itemize}
\end{corollary}

There is an analogous statement for Coulomb branch primary operators: they satisfy $j = -\frac{q}{2}$ together with $Q^-_{(0)} O = 0$, $Q^+_{(1)} O = 0$. The proof is identical.

\begin{proof}
	Lemma \ref{lem:HBP}, together with the definition of a Higgs branch primary operator, gives us the implication $1) \Rightarrow 2)$, so we are left with proving the opposite direction.
	
	Suppose $O$ has $j = \frac{q}{2}$. Lemma \ref{lem:BPSboundprimaries} implies that $O$ is a raviolo Virasoro primary and a superconformal current algebra primary, so we are left with checking its OPEs with $Q^\pm$. The most general form of these OPEs is
	\be
		Q^\pm(z) O(w) \sim \sum_{n \geq 0} \Omega^n_{z-w} (Q^\pm_{(n)} O)(w)
	\ee
	Note that $Q^\pm_{(n)} O$ has spin $\frac{1}{2}(q+1)-n$ and $S$-charge $q\pm 1$. The BPS bound then implies $Q^+_{(n)} O$ vanishes for $n > 0$ and $Q^-_{(n)} O$ vanishes for $n > 1$, whereas a Higgs branch primary further satisfies $Q^\pm_{(0)} O = 0$ and $Q^-_{(1)} O = 0$.
\end{proof}

\subsection{Superconformal flavor symmetries}
\label{sec:HCsymmetries}

Suppose the $N=2$ superconformal raviolo vertex algebra $\CV$ satisfies the BPS bound and has a Hamiltonian $\fg$ symmetry, cf. Definition 4.3.2 of \cite{GarnerWilliams}, generated by fields $\mu_a$ of $R$-charge 1, spin $1$, and $S$-charge 0. Lemma \ref{lem:HBP} immediately implies $\mu_a$ cannot be either a Higgs branch or Coulomb branch primary. However, $\mu_a$ \emph{can} be the superpartner of such a primary: we say that the $\mu_a$ generate a \emph{Higgs branch flavor symmetry} if the $\mu_a$ are the superpartners of Higgs branch primary operators $M_a$ (necessarily having $R$-charge 1, spin $\frac{1}{2}$ and $S$-charge $1$):
\be
\begin{aligned}
	Q^+(z)M_a(w) \sim 0 \qquad Q^+(z)\mu_a(w) \sim \Omega_{z-w}^1 M_a(w) + \Omega_{z-w}^0 \pd M_a(w)\\
	Q^-(z) M_a(w) \sim \Omega_{z-w}^0 \mu_a(w) \qquad  Q^-(z) \mu_a(w) \sim 0 \hspace{1.5cm}
\end{aligned}
\ee
Similarly, we say the $\mu_a$ generate a \emph{Coulomb branch flavor symmetry} if they are the superpartners of Coulomb branch primary operators (necessarily having $R$-charge 1, spin $\frac{1}{2}$, and $S$-charge $-1$).

Our first result concerning Higgs branch flavor symmetries is the following.
\begin{lemma}
	\label{lem:HBsymmprimary}
	Let $\CV$ be an $N=2$ superconformal raviolo vertex algebra satisfying the BPS bound, and suppose the fields $\mu_a$ generate a $\fg$ Higgs branch flavor symmetry of $\CV$, then any Higgs branch primary $O$ is a primary for this current algebra. Moreover, the superpartner $\Psi_O$ is a current algebra primary transforming in the same representation.
\end{lemma}

\begin{proof}
	The first assertion follows from the BPS bound: the coefficient of $\Omega^n_{z-w}$ in the OPE of $\mu_a$ and $O$ has spin $j-n$ and $S$-charge $q$ and therefore must vanish unless $n = 0$, whence $O$ is a current algebra primary. The fact that $\Psi_O$ is a current algebra primary transforming in the same representation as $O$ follows from the regularity of the OPE $\mu_a Q^- \sim 0$ together with associativity.
\end{proof}

If we apply this lemma to the Higgs branch primaries $M_a$ whose superpartners $\mu_a$ generate the symmetry, we see that the $\mu_a$ are themselves primaries and hence they must have a vanishing level, i.e. Higgs branch flavor symmetries are never anomalous.

\begin{corollary}
	Let $\CV$ be an $N=2$ superconformal raviolo vertex algebra satisfying the BPS bound and suppose $\mu_a$ generates a $\fg$ Higgs branch flavor symmetry, then its level must vanish.
\end{corollary}

\subsection{Examples of $N=2$ superconformal raviolo vertex algebras}
\label{sec:SCexamples}
In this final subsection we consider some simple examples of $N=2$ superconformal raviolo vertex algebras. The first examples come from $HT$-twisted free three-dimensional $\CN=4$ theories, those of a free hypermultiplet and a free vector multiplet. The final example is somewhat surprising: we find that the $HT$ twist of a free chiral multiplet (of $R$-charge $\frac{1}{2}$) has an $N=2$ superconformal structure.%
\footnote{Physically, part of this structure arises from the fact that a free chiral multiplet can be deformed a mass term. In the $HT$ twist, this same data can be used to deform to perturbative Chern-Simons theory by adding a term to the Lagrangian of the form $\eta \pd \eta$. This latter deformation doesn't make sense in the physical theory, but does in the ``twisted'' version of the theory described in \cite{AganagicCostelloMcNamaraVafa, CostelloDimofteGaiotto}.} %

\subsubsection{Free hypermultiplet}
Our first example describes local operators in the $HT$ twist of a free hypermultiplet, namely the raviolo vertex algebra $FH = FC^{(1/4)}_{1/2}{}^{\otimes 2}$. We denote the generating fields $Z^\alpha, \psi_\alpha$, $\alpha = 1,2$, and only singular OPEs of these generators are
\be
	Z^\alpha(z) \psi_\beta(w) \sim \Omega^0_{z-w} \delta^\alpha_\beta
\ee
with the remaining OPEs being regular. As described in Section  4.1 of \cite{GarnerWilliams}, this raviolo vertex algebra is conformal with stress tensor
\be
	\Gamma = \tfrac{3}{4}\norm{\psi_\alpha \pd Z^\alpha} - \tfrac{1}{4}\norm{Z^\alpha \pd \psi_\alpha}
\ee
and has a $\fgl(2)$ current subalgebra generated by $\norm{\psi_\alpha Z^\beta}$; we will identify the superconformal current as the diagonal generator
\be
	\sigma = \tfrac{1}{2}\norm{\psi_\alpha Z^\alpha}
\ee
Thus, $Z^\alpha$ is a (bosonic) field of $R$-charge $\tfrac{1}{2}$, spin $\tfrac{1}{4}$, and $S$-charge $\tfrac{1}{2}$; $\psi_\alpha$ is a (fermionic) field of $R$-charge $\tfrac{1}{2}$, spin $\tfrac{3}{4}$, and $S$-charge $-\tfrac{1}{2}$. It is straightforward to verify that the fields
\be
	Q^+ = \tfrac{1}{2}\epsilon_{\beta \alpha}\norm{Z^\alpha \pd Z^\beta} \qquad Q^- = - \tfrac{1}{2} \epsilon^{\beta \alpha}\norm{\psi_\alpha \psi_\beta}
\ee
have the necessary OPEs to generate the $N=2$ superconformal raviolo vertex algebra at central charge $0$, where $\epsilon_{\alpha \beta}$ is the Levi-Civita tensor.%
\footnote{We note that the raviolo vertex subalgebra generated by these fields is actually isomorphic to a quotient of $\SVir^{N=2}$. For example, the normally-ordered product $\norm{Q^- Q^-}$ vanishes in $FH$ due to the fermionic nature of $\psi_\alpha$, but not in $\SVir^{N=2}$. If we consider $N > 1$ copies of this raviolo vertex algebra, we instead find a quotient where the normally-ordered product of $2N$ copies of $Q^-$ must vanish, but no lower power does.} %

The fact that all local operators are realizable as (sums of) normally-ordered products of the generating fields and their derivatives, and the fact that these generating fields satisfy the BPS bound, implies the BPS bound is satisfied for all of $FH$. Moreover, the only operators $O$ saturating the BPS bound $j = \frac{|q|}{2}$ are the trivial local operator $1$, the $Z^\alpha$, and normally-ordered products thereof. There are no operators satisfy $j = - \frac{q}{2}$ except for the identity operator $1 = |0\rangle$, so this is the only Coulomb branch primary operator.

The OPEs of the $Q^\pm$ with the generating fields take the following form: the action on the bosons is given by
\be
	Q^+(z) Z^\alpha(w)  \sim 0 \qquad Q^-(z) Z^\alpha(w) \sim \Omega^0_{z-w} \big(\epsilon^{\alpha \beta}\psi_\beta(w)\big)
\ee
and the action on the fermions is give by
\be
\begin{aligned}
	Q^+(z) \psi_\alpha(w) & \sim \tfrac{1}{2} \Omega^1_{z-w} \big(Z^\beta(w) \epsilon_{\beta \alpha}\big) + \Omega^0_{z-w} \big(\pd Z^\beta(w) \epsilon_{\beta \alpha}\big)\\
	 Q^-(z) \psi_\alpha(w) &\sim 0
\end{aligned}
\ee
It follows that $Z^\alpha$ is a Higgs branch primary operator with superpartner $\Psi_{Z^\alpha} = \psi^\alpha := \epsilon^{\alpha \beta} \psi_\beta$.

\subsubsection{Perturbative free $\CN = 4$ vector multiplet}
The next example we consider describes the algebra of local operators in perturbative pure abelian gauge theory. We consider the raviolo vertex algebra $FV^{\text{pert}} = FC^{(1)}_1 \otimes FC^{(1/2)}_1$; we denote the bosonic generating field of $FC^{(1)}_1$ (resp. $FC^{(1/2)}_1$) $b$ (resp. $\phi$) and the fermionic generating field $c$ (resp. $\lambda$). The singular OPEs of these generating fields are
\be
	b(z) c(w) \sim \Omega^0_{z-w} \qquad \phi(z) \lambda(w) \sim \Omega^0_{z-w}
\ee
with all the remaining OPEs being regular.

This is a conformal raviolo vertex algebra with stress tensor given by
\be
	\Gamma = - \norm{b \pd c} + \tfrac{1}{2} \norm{\lambda \pd \phi} - \tfrac{1}{2} \norm{\phi \pd \lambda}
\ee
We take the superconformal current to be
\be
	\sigma = - \norm{\lambda \phi}
\ee
with the remaining generators given by
\be
	Q^+ = \norm{\lambda \pd c} \qquad Q^- = \norm{b\phi}
\ee
Again, it is straight-forward to verify these fields have the necessary OPEs to realize the $N=2$ superconformal algebra at central charge $0$.%
\footnote{As with the previous example, we find a quotient of $\SVir^{N=2}$ where the normally-ordered product $\norm{Q^+{}^2}=0$. This quotient of $\SVir^{N=2}$ is exchanged with the above by the $\Z_2$ mirror automorphism.} %

Note that $\phi$ has $S$-charge $-1$; $b,c$ have $S$-charge 0; and $\lambda$ has $S$-charge 1; hence the BPS bound is satisfied and the operators $\phi$, $\lambda$, and $c$ saturate the bound. The OPEs of the superconformal generators $Q^\pm$ and the generating fields are given by
\be
\begin{aligned}
	Q^+(z) c(w) & \sim 0 \quad & Q^+(z) b(w) & \sim \Omega^1_{z-w} \lambda(w) + \Omega^0_{z-w} \pd \lambda(w)\\
	Q^-(z) c(w) & \sim \Omega^0_{z-w} \phi(w) \quad & Q^-(z) b(w) & \sim 0\\
\end{aligned}
\ee
and
\be
\begin{aligned}
	Q^+(z) \phi(w) & \sim \Omega^0_{z-w} \pd c(w) \quad & Q^+(z) \lambda(w) & \sim 0\\
	Q^-(z) \phi(w) & \sim 0 \quad & Q^-(z) \lambda(w) & \sim \Omega^0_{z-w} b(w)\\
\end{aligned}
\ee
We see that there are no Coulomb branch primary operators except $1$. The fermions $c$ and $\lambda$ are Higgs branch primary operators with superpartners $\phi$ and $b$, respectively; both have Higgs-branch $R$-charge $1$.

\subsubsection{Free $\CN=4$ vector multiplet}
We can also consider the full, non-perturbative algebra of local operators in pure abelian gauge theory. The relevant raviolo vertex algebra replaces $FC^{(1)}_1$ by the raviolo lattice vertex algebra $\CV_\Z$ described in Section 5.2 of \cite{GarnerWilliams}; this algebra removes $c$, keeping $\pd c = \nu$, and extends the resulting algebra by monopole operators $V_\pm$, with $b = \norm{V_+ \pd V_-}$.%
\footnote{Roughly, local operators in perturbation theory are realized as derived invariants with respect to infinitesimal gauge transformations, i.e. some version of Chevalley-Eilenberg cohomology. True local operators are realized as derived invariants with respect to \emph{finite} gauge transformations; the proper way to do compute these latter derived invariants does not involve a Chevalley-Eilenberg ghost $c$ for the global part of the gauge group, hence we remove $c$ but not its derivatives. See, e.g., Section 6.2.1 of \cite{CostelloDimofteGaiotto} for more details on why one should only consider derivatives of $c$.} %
This non-perturbative algebra is generated by bosons $V_\pm(z)$ and $\phi(z)$ as well as fermions $\nu(z)$ and $\lambda(z)$. 

The $N=2$ superconformal symmetry takes the same form as above so long as we identify $\pd c = \nu$ and $b = \norm{V_+ \pd V_-}$. In particular, we find the following OPEs with the bosonic generators:
\be
\begin{aligned}
	Q^+(z) \nu(w) & \sim 0 \quad & Q^+(z) V_\pm(w) & \sim \pm \Omega^0_{z-w} \norm{\lambda V_\pm}(w)\\
	Q^-(z) \nu(w) & \sim \Omega^1_{z-w} \phi(w) + \Omega^0_{z-w} \pd \phi(w) \quad & Q^-(z) b(w) & \sim 0\\
\end{aligned}
\ee
and
\be
\begin{aligned}
	Q^+(z) \phi(w) & \sim \Omega^0_{z-w} \nu(w) \quad & Q^+(z) \lambda(w) & \sim 0\\
	Q^-(z) \phi(w) & \sim 0 \quad & Q^-(z) \lambda(w) & \sim \Omega^0_{z-w} \norm{V_+ \pd V_-}(w)\\
\end{aligned}
\ee

The only non-trivial Higgs branch primary operator is $\lambda$, with superpartner $b$; no other operators satisfy $j = \frac{q}{2}$. The monopole operators $V_\pm$ all have spin $0$ and $S$-charge $0$, hence saturate the BPS bound. The above OPEs imply $V_\pm$ are Coulomb branch primary operators, with superpartners $\pm\norm{\lambda V_\pm}$. The removal of $c$ implies that $\phi$ is also a Coulomb branch primary operator (with superpartner $\nu$).

We note that $\nu$ is an abelian current generating the \emph{topological flavor symmetry} of the raviolo lattice algebra $\CV_\Z$ and it's a superpartner of the Coulomb branch primary operator $\phi$, we see that it generates a $\fgl(1)$ Coulomb branch flavor symmetry. 

\subsubsection{Free chiral multiplet}
The final example we consider does not arise from the $HT$-twist of a three-dimensional $\CN=4$ theory. In a sense, it is a fermionic/orthogonal counterpart of the symplectic/bosonic free hypermultiplet. We consider $FC^{(3/4)}_{1/2}$ with generating fields $X, \eta$ whose OPE takes the form
\be
	X(z) \eta(w) \sim \Omega^0_{z-w}
\ee
The stress tensor is given by
\be
	\Gamma = \tfrac{1}{4} \norm{\eta \pd X} - \tfrac{3}{4} \norm{X \pd \eta}
\ee
and we take the superconformal current to be $\sigma = -\tfrac{1}{2} \norm{\eta X}$. The remaining superconformal generators are
\be
	Q^+ = \tfrac{1}{2}\norm{\eta \pd \eta} \qquad Q^- = \tfrac{1}{2}\norm{X^2}
\ee
and together these generate a copy of the $N=2$ superconformal raviolo vertex algebra at central charge $0$; we will denote this $N=2$ superconformal raviolo vertex algebra $SFC$.%
\footnote{These fields generate the same quotient of $\SVir^{N=2}$ found in $FV$ and $FV^{\text{pert}}$.} %

The generator $X$ (resp. $\eta$) has $R$-charge $\frac{1}{2}$ (resp. $\frac{1}{2}$), $S$-charge $-\frac{1}{2}$ (resp. $\tfrac{1}{2}$) and spin $\frac{3}{4}$ (resp. $\tfrac{1}{4}$); we see that the BPS bound is satisfied, and $\eta$ saturates the BPS bound. There are no operators with $j = -\frac{q}{2}$ other than the trivial operator $1$, hence the only Coulomb branch primary operator is $1$. The OPEs of the superconformal generators $Q^\pm$ and the generating fields are given by
\be
\begin{aligned}
	Q^+(z) X(w) & \sim \tfrac{1}{2} \Omega^1_{z-w} \eta(w)	+ \Omega^0_{z-w} \pd \eta(w) \qquad & Q^+(z) \eta(w) & \sim 0\\
	Q^-(z) X(w) & \sim 0 \qquad & Q^-(z) \eta(w) & \sim \Omega^0_{z-w} X(w)\\
\end{aligned}
\ee
from which we see that $\eta$ is a Higgs branch primary operator, with superpartner $X$.

\subsection{Raviolo superconformal algebras with more supersymmetry}
\label{sec:moresusy}

So far we have focused on superconformal raviolo vertex algebras associated to $HT$ twists of three-dimensional theories with $N=4$ supersymmetry.
Most importantly for this paper, the object that we call the $N=2$ raviolo superconformal algebra $SVir_{N=2}$ arises as a symmetry of the $HT$ twist of a $\cN=4$ theory.
In this section we comment on superconformal algebras arising in theories with different amounts of supersymmetry.

\subsubsection{$N=1$ superconformal algebra}

We start with an example with less supersymmetry than in the other parts of the paper.
Let $Vir$ be the raviolo Virasoro algebra as above with bosonic generators $\{G_m\}_{m \in \Z}$ and fermionic generators $\{\Gamma_m, \xi\}_{m \in \Z_{\geq 0}}$.
We add to this fermionic generators $\{\theta_m\}$ and bosonic generators $\{Q_m\}$ satisfying
\be
\begin{aligned}[]
	[G_m, \theta_{n}] & = \big(\tfrac{1}{2}m -n\big)\theta_{m+n-1}\\
	[G_m, Q_n] & = \begin{cases}
		0 & n + 1 < m\\
		\big(\tfrac{1}{2}m+n+1\big)Q_{n-m+1} & n +1 \geq m\\
	\end{cases}
\end{aligned}
\ee
\be
\begin{aligned}[]
	[\Gamma_m, \theta_{n}] &= \begin{cases}
		0 & n > m+1\\
		-\big(\tfrac{1}{2}(m+1)+n\big)Q_{m-n+1} & n \leq m+1\\
	\end{cases}\\
	[\Gamma_m, Q_{n}] & = 0
\end{aligned}
\ee
\be
\begin{aligned}[]
	[\theta_m, \theta_n] = 2G_{m+n} \qquad [Q^\pm_m, Q^\mp_n] = 0\\
	[\theta_m, Q_n] = \begin{cases} 
		0 & n+2 < m\\
		-\frac{m(m-1)}{3}\xi & n+2 = m\\
		-2\Gamma_{n-m} & n \geq m\\
	\end{cases} \hspace{-0.25cm}
\end{aligned}
\ee
We call this Lie algebra the \emph{raviolo $N=1$ superconformal algebra} and denote it $SVir^{N=1}$.

We can use this Lie algebra to construct a (universal) $N=1$ superconformal raviolo vertex algebra via induction, just as we did above:
\be
	\SVir^{N=1} = U SVir^{N=1} \otimes_{USVir^{N=1}_+} \C
\ee
As $\xi$ is central and acts as $0$ on $\C$, it follows that $\xi$ acts as $0$ on all of $\SVir^{N=1}$; the resulting raviolo vertex algebra is analogous to the super Virasoro vertex algebra at vanishing central charge.
As above there is a universal version $\SVir^{N=1}_{univ}$ which incorporates the central term $\xi$.
We can realize $\SVir^{N=1}$ as the quotient of $\SVir^{N=1}_{univ}$ by vectors proportional to $\xi$. 

In addition to the raviolo Virasoro field $\Gamma(z)$, the image of the linear state $Q_0 \otimes 1$ corresponds to the bosonic field of spin $\tfrac{3}{2}$:
\be
	Q (z) = \sum_{n \geq 0} z^{n} Q_{n} + \Omega^n_z \theta_{n}\,.
\ee
As in the $N=2$ case it is immediate to see that $\SVir^{N=1}$ and $\SVir^{N=1}_{univ}$ are raviolo vertex algebras with spin grading.
The OPEs of bosonic field $Q(z)$ transforms as with the stress tensor are
\be
	\Gamma(z) Q(w) \sim \tfrac{3}{2} \Omega^1_{z-w} Q(w) + \Omega^0_{z-w} \pd Q(w) .
\ee
corresponding to the fact that $Q$ is a raviolo Virasoro primary of spin $\frac{3}{2}$. The OPE of $Q$ with itself is
\be
	Q(z) Q(w) \sim \Omega_{z-w}^2 (-2\xi/3) - \Omega_{z-w}^0 2\Gamma(w) .
\ee

There is a morphism of raviolo vertex algebras 
\beqn
	\SVir^{N=1}_{univ} \to \SVir^{N=2}_{univ}
\eeqn
which is the identity on the generators $G_m, \Gamma_m$ and maps $\theta_m \mapsto \theta_m^+ + \theta_m^-$ and $Q_m \mapsto Q_m^+ + Q_m^-$.

We show in \cite{HTenhance} that $\SVir^{N=1}$ is present in the algebra of local operators in the $HT$ twist of any $\cN=3$ superconformal Chern--Simons-matter theory. More generally, we expect this raviolo vertex algebra is realized as symmetries of the $HT$ twist of theories with $\cN = 3$ supersymmetry. We note that the bosonic generators $G_{0}$, $G_{1}$, and $G_2$ together with the fermionic generators $\theta_0$, $\theta_1$ generate a copy of $\lie{osp}(1|2)$, identified with the $HT$-twist of the $\CN=3$ superconformal algebra $\lie{osp}(3|4)$.

\subsubsection{$N=3$ superconformal algebra}
We now add more supersymmetry. Rather than write down explicit commutators of the resulting Lie superalgebra, we will instead write down the corresponding OPEs; the commutators can be obtained by computing residues.

We add to the raviolo Virasoro field $\Gamma(z)$ (at central charge $\xi$) fermionic fields $\sigma_A(w)$ linear in $A \in \lie{sl}(2)$, i.e. $\sigma_{A+B} = \sigma_A + \sigma_B$. These fields are raviolo Virasoro primaries of spin 1
\be
	\Gamma(z) \sigma_A(w) \sim \Omega^1_{z-w} \sigma_A(w) + \Omega^0_{z-w} \pd \sigma_A(w)
\ee
and generate an $\lie{sl}(2)$ raviolo current algebra at level $\xi/3$
\be
	\sigma_A(z) \sigma_B(w) \sim \Omega^1_{z-w}  \big(\xi/3 \Tr(AB)\big) + \Omega^0_{z-w} \sigma_{[A,B]}(w) \,.
\ee
We use $\Tr$ to denote the trace in the standard representation of $\lie{sl}(2)$.
We then add bosonic fields $Q_A$, also linear in $A \in \lie{sl}(2)$, that are raviolo Virasoro primaries of spin $\frac{3}{2}$
\be
	\Gamma(z) Q_A(w) \sim \tfrac{3}{2} \Omega^1_{z-w} Q_A(w) + \Omega^0_{z-w} \pd Q_A(w)
\ee
and transform as raviolo current algebra primaries transforming in the adjoint representation
\be
	\sigma_A(z) Q_B(w) \sim \Omega^0_{z-w} Q_{[A,B]}(w) \,.
\ee
Finally, the OPE of these bosonic generators takes the following form:
\be
\begin{aligned}
	Q_A(z) Q_B(w) & \sim -\Omega^2_{z-w} \big(\Tr(AB) \xi/3\big) - \Omega^1_{z-w} \sigma_{[A,B]}(w)\\
	& \qquad - \Omega^0_{z-w} \bigg(\Tr(AB) \Gamma(w) + \tfrac{1}{2}\partial \sigma_{[A,B]}(w)\bigg) \,.
\end{aligned}
\ee

We call underlying Lie superalgebra the \emph{raviolo $N=3$ superconformal algebra} and denote it $SVir^{N=3}$.
The universal $N=3$ superconformal raviolo vertex algebra is
\be
	\SVir^{N=3}_{univ} = U SVir^{N=3} \otimes_{USVir^{N=3}_+} \C[\xi]
\ee
We denote by $\SVir^{N=3}$ the quotient of $\SVir^{N=3}_{univ}$ by vectors proportional to $\xi$. 
There are many maps from $\SVir^{N=2}$ to $\SVir^{N=3}$, parameterized by an choice of embedding the $\lie{gl}(1) \hookrightarrow \lie{sl}(2)$ describing the image of $\sigma$ in the above $\fsl(2)$ current algebra.

This raviolo vertex algebra encodes the symmetries of the $HT$ twist of theory with $\cN=5$ supersymmetry.
It follows from our work in \cite{HTenhance}, that $\SVir^{N=3}$ is present in the algebra of local operators of the $HT$ twist of any $\cN=5$ superconformal Chern--Simons-matter theory. We note that the bosonic generators $G_0$, $G_1$, $G_2$ and $S_{A,0}$ together with the fermions $\theta_{A,0}$ and $\theta_{A,1}$ generate a copy of $\lie{osp}(3|2)$, identified as the $HT$-twist of the $\CN=5$ superconformal algebra $\lie{osp}(5|4)$.

\subsubsection{(Big) $N=4$ superconformal algebra}

The next example of superconformal algebra is realized from twisted three-dimensional $\CN=6$ supersymmetry. We will again write down OPEs of the generating raviolo fields instead of commutators.

The most general form depends on a single complex parameter $p \in \C$. In addition to the raviolo Virasoro field $\Gamma(z)$ of central charge $\xi$, there are two commuting $\lie{sl}(2)$ raviolo currents $\sigma^\pm_A$
\be
\begin{aligned}
	\sigma^\pm_A(z) \sigma^\pm_B(w) & \sim \Omega^1_{z-w} \big(\kappa^\pm \Tr(AB)\big) + \Omega^0_{z-w} \sigma^\pm_{[A,B]}(w)\\ \sigma^\pm_A(z) \sigma^\mp_B(w) & \sim 0
\end{aligned}
\ee
as well as an abelian raviolo current $\upsilon$ commuting with the two of them
\be
	\upsilon(z) \upsilon(w) \sim \Omega^1_{z-w} \kappa \qquad  \sigma^\pm_A(z) \upsilon(w) \sim 0
\ee
where the central generators $\kappa$, $\kappa^\pm$, and $\xi$ satisfy
\be
	\kappa^\pm = (1\pm p)\kappa \qquad \xi = (1-p^2) \kappa
\ee
All of these generators are raviolo Virasoro primaries of spin $1$. The remaining generators are bosonic and come in two types: there are the bosonic fields $Q_{\lambda}(z)$ and $P_{\lambda}(w)$ linear in $\lambda \in \C^2_+ \otimes \C^2_-$ that are raviolo Virasoro primaries of spins $\frac{3}{2}$ and $\frac{1}{2}$, respectively
\be
\begin{aligned}
	\Gamma(z) Q_\lambda(w) & \sim \tfrac{3}{2} \Omega^1_{z-w} Q_\lambda(w) + \Omega^0_{z-w} \pd Q_\lambda(w)\\
	\Gamma(z) P_\lambda(w) &\sim \tfrac{1}{2} \Omega^1_{z-w} P_\lambda(w) + \Omega^0_{z-w} \pd P_\lambda(w)
\end{aligned}
\ee
The operators $P_\lambda(z)$ are current algebra primaries transforming in the representation $(\C^2_+ \otimes \C^2_-)_0$ of $\lie{sl}(2)_+ \oplus \lie{sl}(2)_- \oplus \lie{gl}(1)$, where $\C^2_\pm$ is the standard/fundamental representation of $\lie{sl}(2)_\pm$ and the subscript ${}_0$ indicates $\lie{gl}(1)$ weight $0$:
\be
	\sigma^\pm_A(z) P_\lambda(w) \sim \Omega^0_{z-w} P_{A^\pm \lambda}(w) \qquad \upsilon(z) P_\lambda(w) \sim 0
\ee
where $A^+ = A \otimes \id_{\C^2_-}$ and $A^- = \id_{\C^2_+} \otimes A$. The generators $Q_\lambda(z)$ are no longer current algebra primaries, but are modulo $P_\lambda$:
\be
\begin{aligned}
	\sigma^\pm_A(z) Q_\lambda(w) & \sim \Omega^1_{z-w} \bigg(\mp\tfrac{1}{4}(1\pm p) P_{A^\pm \lambda}(w)\bigg) + \Omega^0_{z-w} Q_{A^\pm \lambda}(w)\\
	\upsilon(z) Q_\lambda(w) & \sim \Omega^1_{z-w} P_{\lambda}
\end{aligned}
\ee

To express the fermionic OPEs, we need some notation. For $\lambda_1, \lambda_2 \in \C^2_+ \otimes \C^2_-$, we let $(\lambda_1, \lambda_2)$ denote the $\fsl(2)_\pm$-invariant pairing of $\C^2_+ \otimes \C^2_-$ with itself:
\be
	(\lambda_1, \lambda_2) = \epsilon_{\alpha_1 \alpha_2} \epsilon_{\dot{\alpha}_1 \dot{\alpha}_2} \lambda_1^{\alpha_1 \dot{\alpha}_1}\lambda_2^{\alpha_2 \dot{\alpha}_2}
\ee
We similarly let $[\lambda_1, \lambda_2]_+$ and $[\lambda_1, \lambda_2]_-$ be the two traceless matrices with matrix elements
\be
	\epsilon_{\dot{\alpha}_1 \dot{\alpha}_2} (\delta^{\alpha}{}_{\alpha_1}\epsilon_{\alpha_2 \beta} + \delta^{\alpha}{}_{\alpha_2}\epsilon_{\alpha_1 \beta}) \lambda_1^{\alpha_1 \dot{\alpha}_1}\lambda_2^{\alpha_2 \dot{\alpha}_2}
\ee
and
\be
	\epsilon_{\alpha_1 \alpha_2} (\delta^{\dot{\alpha}}{}_{\dot{\alpha}_1} \epsilon_{\dot{\alpha}_2 \dot{\beta}} + \delta^{\dot{\alpha}}{}_{\dot{\alpha}_2} \epsilon_{\dot{\alpha}_1 \dot{\beta}}) \lambda_1^{\alpha_1 \dot{\alpha}_1}\lambda_2^{\alpha_2 \dot{\alpha}_2}
\ee
respectively. The OPEs of these bosonic generators then take the following form. First, the OPEs involving at least one copy of $P_\lambda$ take the form
\be
	P_{\lambda_1}(z) P_{\lambda_2}(w) \sim \Omega^0_{z-w} \big(\tfrac{8}{3}(\lambda_1, \lambda_2)\kappa\big)
\ee
and
\be
	P_{\lambda_1}(z) Q_{\lambda_2}(w) \sim \Omega^0_{z-w} \bigg(-\sigma^+_{[\lambda_1, \lambda_2]_+}(w) + \sigma^-_{[\lambda_1, \lambda_2]_-}(w) - \tfrac{1}{2}(\lambda_1, \lambda_2) \upsilon(w)\bigg)
\ee
Finally, the OPE of $P_{\lambda_1}$ and $P_{\lambda_2}$ is given by
\be
\begin{aligned}
	Q_{\lambda_1}(z) Q_{\lambda_2}(w) &\\ 
	& \hspace{-2cm} \sim -\Omega^2_{z-w} \bigg(\tfrac{1}{3}(\lambda_1, \lambda_2) \xi\bigg)\\
	& \hspace{-1.75cm} -\Omega^1_{z-w} \bigg(\tfrac{1}{2}(1-p)\sigma^+_{[\lambda_1, \lambda_2]_+}(w) + \tfrac{1}{2}(1+p)\sigma^-_{[\lambda_1, \lambda_2]_-}(w)\bigg)\\
	& \hspace{-1.75cm} - \Omega^0_{z-w} \bigg((\lambda_1, \lambda_2) \Gamma(w) + \tfrac{1}{4}(1-p)\pd\sigma^+_{[\lambda_1, \lambda_2]_+}(w) + \tfrac{1}{4}(1+p)\pd\sigma^-_{[\lambda_1, \lambda_2]_-}(w)\bigg)\\
\end{aligned}
\ee

We call underlying Lie superalgebra the \emph{raviolo big $N=4$ superconformal algebra} and denote it $SVir^{N=4}_p$.
The universal (big) $N=4$ superconformal raviolo vertex algebra is
\be
	\SVir^{N=4}_{p,univ} = U SVir^{N=4}_p \otimes_{USVir^{N=4}_{p,+}} \C[\kappa]
\ee
We denote by $\SVir^{N=4}_p$ the quotient of $\SVir^{N=4}_{p,univ}$ by vectors proportional to~$\kappa$. 
This is a raviolo analog of the big superconformal vertex algebra characterized in \cite{KV,STV}.

We note that the bosonic generators $G_0$, $G_1$, $G_2$ and $S^\pm_{A,0}$ together with the fermions $\theta_{\lambda,0}$ and $\theta_{\lambda,1}$ realize the exceptional Lie superalgebra $D(2,1;\frac{1-p}{1+p})$. In our work \cite{HTenhance}, we find that the special case $p = 0$ is realized as symmetries of a family of $\CN=6$ Chern-Simons-matter theories; of course $D(2,1;1) \simeq \lie{osp}(4|2)$, which we identify with the $HT$ twist of the $\CN=6$ superconformal algebra $\lie{osp}(6|4)$.

\subsubsection{An exceptional superconformal algebra}

The final example is realized from twisted three-dimensional $\cN=8$ supersymmetry.
This example is somewhat different in that there is no (shifted) central charge.

In addition to the raviolo Virasoro field $\Gamma(z)$ (of vanishing central charge), there is an $\lie{sl}(4) = \lie{so}(6)$ raviolo currents $\sigma_A (z)$, $A \in \lie{sl}(4)$, at level zero:
\beqn
\sigma_A(z) \sigma_B(w) \sim \Omega^0_{z-w} \sigma_{[A,B]} (w) . 
\eeqn
These fields are raviolo Virasoro primaries of the usual spin $1$.

The remaining generators are bosonic and come in two types.
First, there are bosonic fields $Q_B(z)$ labeled by a skew-symmetric $4 \times 4$ matrix $B$, each of which are raviolo Virasoro primaries of spin $\frac32$.
These obey the following self-OPE's:
\beqn
\begin{aligned}
	Q_A(z) Q_B(w) & \sim - \Omega^1_{z-w} \sigma_{(A \star B)_0} (w) \\
	& - \Omega^0_{z-w} \left( \tfrac{1}{2}\text{tr}(A \star B) \Gamma(w) + \tfrac12 \partial \sigma_{(A \star B)_0} (w) \right)
\end{aligned}
\eeqn
In this expression the skew-symmetric matrix $\star A$ is defined by $(\star A)_{IJ} = \frac12 \epsilon_{IJKL} A^{KL}$.
Also, $(A \star B)_0$ is the traceless part of the matrix $A \star B$.

The second type of bosonic field is denoted $P_C(z)$ which is labeled by a symmetric $4 \times 4$ matrix $C$, each of which are raviolo Virasoro primaries of spin $\frac12$.
They have a non-singular OPE with each other
\beqn
P_A (z) P_B(w) \sim 0
\eeqn
and the following OPE with the other bosonic fields
\beqn
Q_A(z) P_B(w) \sim -\tfrac12 \Omega_{z-w}^0 \sigma_{B \star A} (w) .
\eeqn

Finally, the OPEs of the bosonic fields with the current algebra are as follows. As mentioned above, both $P_A$ and $Q_A$ are raviolo Virasoro primaries (of spins $\tfrac{1}{2}$ and $\tfrac{3}{2}$, respectively). The fields $P_A$ are raviolo current algebra primaries (transforming in the representation $\Sym^2 \C^4$)
\beqn
	\sigma_A(z) P_B(w) \sim \Omega^0_{z-w} P_{AB + BA^T} (w).
\eeqn
but, as in $N = 4$, the $Q_A$ are only raviolo current algebra primaries (transforming in the representation $\bigwedge^2 \C^4$) modulo $P_A$
\beqn
	\sigma_A(z) Q_B(w) \sim \Omega^1_{z-w} P_{A B - B A^T} + \Omega^{0}_{z-w} Q_{AB + BA^T} (w). \\
\eeqn
Note that the antisymmetry of $B$ in the second equation implies $A B - B A^T$ is symmetric.

We call underlying Lie superalgebra the \emph{raviolo exceptional superconformal algebra} and denote it $EVir$.
The Lie superalgebra of positive modes of this superalgebra is the exceptional Lie superalgebra $E(1|6)$ classified in \cite{CKstructure}.
The universal exceptional superconformal raviolo vertex algebra is
\be
	\text{EVir}_{univ} = U EVir \slash UEVir_+ 
\ee

This is the raviolo analog of the vertex algebra called $CK_6$ defined in \cite{CK}.
In fact, the OPE's above are complete raviolo translations of the OPE's describing $CK_6$.
It was pointed out in \textit{loc. cit.} that part of the exceptional nature of the vertex algebra $CK_6$ (and the corresponding mode algebra) is that it does not admit a central charge.
For this reason, it is expected that $CK_6$ (or its annihilation subalgebra $E(1|6)$) does not appear as the symmetries of some two-dimensional conformal field theory.

In our work \cite{HTenhance}, we argue that $E(1|6)$, and its raviolo vertex algebra enhancement $\text{EVir}$, appear as the symmetries of twists of particular three-dimensional $\cN = 8$ superconformal field theories.
Indeed, we show that $\text{EVir}$ is realized as symmetries of the $HT$ twists of the so-called BLG \cite{BL1,BL2,Gustavsson} and rank 1 ABJM \cite{ABJM, ABJ} theories at levels $k=1,2$.

\section{Higgs and Coulomb Branches}
\label{sec:HC}
Based on the aforementioned analogy with two-dimensional superconformal QFT, as well as their relevance in three-dimensional $\CN=4$ QFT, we now discuss some important properties of Higgs and Coulomb branch primary operators. We will focus on Higgs branch primary operators in the following; analogous results hold for Coulomb branch primary operators.

\subsection{Higgs and Coulomb branch chiral rings}
\label{sec:chiralrings}
Our first result says that we can equip the vector space of Higgs branch primary operators with the structure of a commutative, associative algebra, where the product is given by collision, i.e. operator product together with taking the coincidence limit $z \to w$.

\begin{proposition}
	\label{prop:chiralOPE}
	Let $\CV$ be an $N=2$ superconformal raviolo vertex algebra satisfying the BPS bound. The OPE of two Higgs branch primary operators $O_1$ and $O_2$ is necessarily regular $O_1 O_2 \sim 0$ and, moreover, their normal ordered product $\norm{O_1 O_2}$ is also a Higgs branch primary operator, with superpartner
	\begin{equation*}
		\Psi_{\norm{\,O_1 O_2\,}} = \norm{\Psi_{O_1} O_2} + (-1)^{|O_1|} \norm{O_1 \Psi_{O_2}}
	\end{equation*}
\end{proposition}

\begin{proof}
	Denote the spins and $S$-charges of $O_1, O_2$ by $j_1, j_2$ and $q_1, q_2$ and write the singular terms of the OPE as
	\be
		O_1(z) O_2(w) \sim \sum_n \Omega^n_{z-w} (O_{1,(n)} O_2)(w)
	\ee
	The operator $(O_{1,(n)} O_2)$ has spin $j_1 + j_2-n-1$ and $S$-charge $q_1 + q_2$. Lemma \ref{lem:HBP} implies $q_1 = 2 j_1 \geq 0$ and $q_2 = 2 j_2 \geq 0$, and since $q_1 + q_2 = 2(j_1 + j_2) > 2 (j_1+j_2-n-1)$ for all $n \geq 0$, the BPS bound implies $(O_{1,(n)} O_2) = 0$ and so the OPE of $O_1$, $O_2$ is regular.
	
	To show that $\norm{O_1 O_2}$ is a Higgs branch primary operator, it suffices to check its OPEs with $Q^\pm$: Propositions 4.3.4 and 4.4.4 of \cite{GarnerWilliams} imply $\norm{O_1 O_2}$ is a raviolo Virasoro primary (of spin $j_1 + j_2$) and current algebra primary (of $S$-charge $q_1 + q_2$). It is straightforward to compute these OPEs using Corollary 2.2.5 of loc. cit.:
	\be
	\begin{aligned}
		Q^+(z) \norm{O_1 O_2}(w) & \sim 0\\
		Q^-(z) \norm{O_1 O_2}(w) & \sim \Omega^0_{z-w} \big(\norm{\Psi_{O_1} O_2}(w) + (-1)^{|O_1|} \norm{O_1 \Psi_{O_2}}(w)\big)
	\end{aligned}
	\ee
\end{proof}

We call this the \emph{Higgs branch chiral ring} of $\CV$ and denote it $A_H[\CV]$. Proposition 2.2.6 of \cite{GarnerWilliams} implies $A_H[\CV]$ is a commutative, associative algebra, and we view it as the ring of functions on an affine (possibly non-reduced, super) scheme that we call the \emph{Higgs branch} $\CM_H[\CV]$. Similarly, collision of Coulomb branch primary operators gives a second commutative, associative algebra we call the \emph{Coulomb branch chiral ring} $A_C[\CV]$, viewed as the ring of functions on the \emph{Coulomb branch} $\CM_C[\CV]$. These chiral rings are the superconformal raviolo vertex algebra analog of the chiral and twisted chiral rings of a 2d $\CN=(2,2)$ superconformal QFT.

It is convenient to define a \emph{Higgs branch $R$-charge} $R_B$ and \emph{Higgs branch spin} $J_B$ via 
\be
	R_B = R + \sigma_{(0)} \qquad J_B = J - \tfrac{1}{2}\sigma_{(0)}
\ee
We can think of this as modifying the original $R$-charge grading on $\CV$ by the $S$-charge together with shifting the stress tensor to $\Gamma_B = \Gamma + \frac{1}{2}\pd \sigma$.

The two-dimensional chiral rings have yet more structure: they have a natural degree $-1$ Poisson bracket, also called a Gerstenhaber bracket, cf. \cite{LianZuckerman} (see also \cite{WittenZwiebach, PenkavaSchwarz, WuZhu}). We now show that the Higgs branch chiral ring is similarly equipped with such a bracket. Given a fully-fledged chain-level version of raviolo vertex algebras, we expect this to be enhanced to a full $E_3$-structure, cf. \cite{descent}.

\begin{theorem}
	\label{thm:poissonHB}
	Let $\CV$ be an $N=2$ superconformal raviolo vertex algebra satisfying the BPS bound and let $O_1, O_2$ be Higgs branch primary operators, then
	\begin{equation*}
		\{O_1 , O_2\}(z) := \tfrac{1}{2}\oint_z \diff x \bigg(\Psi_{O_1}(x) O_2(z) - (-1)^{|O_1||O_2|} \Psi_{O_2}(x) O_1(z)\bigg)
	\end{equation*}
	defines a Poisson bracket on the Higgs branch chiral ring. This bracket has intrinsic Higgs branch $R$-charge $-2$.
\end{theorem}

\begin{proof}
	We start by showing $\{O_1, O_2\}$ is indeed a Higgs branch primary operator. First, it saturates the BPS bound: it has spin $j_1 + j_2 - \frac{1}{2}$ and $S$-charge $q_1 + q_2 - 1 = 2(j_1 + j_2-\frac{1}{2})$. Lemma \ref{lem:BPSboundprimaries} implies $\{O_1, O_2\}$ is a both a raviolo Virasoro primary and a superconformal current algebra primary. It therefore suffices to check the OPEs of $\{O_1, O_2\}$ with $Q^\pm$, cf. Corollary \ref{cor:HBP2}.
	
	The regularity of the OPE of $Q^+$ with $\{O_1, O_2\}$ can be seen by realizing the two summands of $\{O_1, O_2\}$ independently have regular OPEs with $Q^+$: 
	\be
	\begin{aligned}
		Q^+(z) \bigg(\oint_w \diff x \Psi_{O_1}(x) O_2(w)\bigg) & \sim \oint_w \diff x \big(\Omega^1_{z-x} (q_1 O_1(x) + \Omega^0_{z-x} \pd O_1(x)\big)O_2(w)\\
		& \sim q_1 \sum_{l \geq 0} (-1)^{l+1}(l+1)\Omega^{l+1}_{z-w} \oint_w \diff x (x-w)^l O_1(x) O_2(w)\\
		& \qquad + \sum_{l \geq 0}(-1)^{l+1} \Omega^l_{z-w} \oint_w \diff x (x-w)^l \pd O_1(x) O_2(w)\\
		& \sim 0
	\end{aligned}
	\ee
	where we used regularity of the OPE between $O_1$ and $O_2$ established in Proposition \ref{prop:chiralOPE}. The OPE with $Q^-$ is the easier of the two: using $Q^- \Psi_{O_1} \sim 0 \sim Q^- \Psi_{O_2}$ we get
	\be
	\begin{aligned}
		Q^-(z) \{O_1, O_2\}(w) & \sim \tfrac{1}{2}\oint_w \diff x\Psi_{O_1}(x) \Omega^0_{z-w} \Psi_{O_2}(w)\\
		& -\tfrac{1}{2}(-1)^{|O_1||O_2|} \oint_w \diff x \Psi_{O_2}(x) \Omega^0_{z-w} \Psi_{O_1}(w)\\
		& \sim \tfrac{1}{2} (-1)^{|O_1|} \Omega^0_{z-w} \oint_w \diff x \Psi_{O_1}(x) \Psi_{O_2}(w)\\
		& + \tfrac{1}{2} (-1)^{|O_1|+(|O_1|+1)(|O_2|+1)} \Omega^0_{z-w} \oint_w \diff x \Psi_{O_2}(x) \Psi_{O_1}(w)
	\end{aligned}
	\ee
	so that $\{O_1, O_2\}$ is indeed a Higgs branch primary operator with
	\be
	\label{eq:superpartnerbracket}
		\Psi_{\{O_1,O_2\}}(z) = \tfrac{1}{2} (-1)^{|O_1|} \oint_z \diff x \bigg(\Psi_{O_1}(x) \Psi_{O_2}(z) + (-1)^{(|O_1|+1)(|O_2|+1)} \Psi_{O_2}(x) \Psi_{O_1}(z)\bigg)
	\ee
	
	It is clear from the definition that the bracket has the necessary skew-symmetry. The fact that it is a derivation of the product can be seen as follows. Corollary 2.2.5 of \cite{GarnerWilliams} implies the following identity:
	\be
	\begin{aligned}
		\oint_w \diff x \Psi_{O_1}(x) \norm{O_2 O_3}(w) & = \norm{\bigg(\oint_w \diff x \Psi_{O_1}(x) O_2(w)\bigg) O_3(w)} \\
		& + (-1)^{|O_1| |O_2|}\norm{O_2(w) \bigg(\oint_w \diff x \Psi_{O_1}(x) O_3(w)\bigg)}
	\end{aligned}
	\ee
	The expression for $\Psi_{\norm{\,O_2 O_3\,}}$ in Proposition \ref{prop:chiralOPE} tells us
	\be
	\oint_w \diff x \Psi_{\norm{\,O_2 O_3\,}}(x) O_1(w) = \oint_w \diff x \big(\norm{\Psi_{O_2}O_3}(x) + (-1)^{|O_2||O_3|} \norm{\Psi_{O_3}O_2}(x)\big)O_1(w)
	\ee
	The BPS bound implies that the OPE of $\norm{\Psi_{O_2}O_3}$ and $O_1$ has no terms proportional to $\Omega^n_{z-w}$ for $n > 0$, whence
	\be
	-(-1)^{|O_1|(|O_2|+|O_3|)}\oint_w \diff x \norm{\Psi_{O_2}O_3}(x) O_1(w) = -(-1)^{|O_1|}\oint_w \diff x O_1(x) \norm{\Psi_{O_2}O_3}(w)
	\ee
	Corollary 2.2.5 of \cite{GarnerWilliams}, Proposition \ref{prop:chiralOPE}, and the BPS bound then give us
	\be
	\begin{aligned}
		-(-1)^{|O_1|}\oint_w \diff x O_1(x) \norm{\Psi_{O_2}O_3}(w) & = (-1)^{|O_1|}\norm{\bigg(\oint_w \diff x O_1(x) \Psi_{O_2}(w) \bigg)O_3(w)}\\
		& = -(-1)^{|O_1||O_2|} \norm{\bigg(\oint_w \diff x \Psi_{O_2}(x) O_1(w) \bigg)O_3(w)}\\
	\end{aligned}
	\ee
	Together with a similar analysis of the OPE $\norm{\Psi_{O_3} O_2}$ and $O_1$, these imply 
	\be
	\begin{aligned}
		\{O_1, \norm{O_2 O_3}\}(w) & = \oint_w \diff x \Psi_{O_1}(x) \norm{O_2 O_3}(w) - (-1)^{|O_1|(|O_2|+|O_3|)} \Psi_{\norm{\, O_2 O_3\,}}(x) O_1(w)\\
		& =  \norm{\{O_1, O_2\} O_3}(w) + (-1)^{|O_1||O_2|}\norm{O_2 \{O_1, O_3\}}(w)
	\end{aligned}
	\ee
	
	We are left with showing the bracket satisfies the Jacobi identity. Using the above formula for $\Psi_{\{O_1,O_2\}}$, we have the following identity:
	\be
	\begin{aligned}
		\{\{O_1,O_2\}, O_3\}(w) & = \tfrac{1}{4}(-1)^{|O_1|} \oint_w \diff x \oint_x \diff z \Psi_{O_1}(z) \Psi_{O_2}(x) O_3(w)\\
		& + \tfrac{1}{4} (-1)^{(|O_1|+1)(|O_2|+1)+|O_1|} \oint_w \diff x \oint_x \diff z \Psi_{O_2}(z) \Psi_{O_1}(x) O_3(w)\\
		& - \tfrac{1}{4}\oint_w \diff x \Psi_{O_3}(x)\oint_w \diff z \Psi_{O_1}(z) O_2(w)\\
		& + \tfrac{1}{4} (-1)^{|O_1||O_2|} \oint_w \diff x \Psi_{O_3}(x) \oint_w \diff z \Psi_{O_2}(z) O_1(w)\\
	\end{aligned}
	\ee
	The first two lines integrate $z$ over a small 2-sphere surrounding $x$ and $x$ over a larger 2-sphere surround $w$; the third and fourth lines integrate $z$ over a small sphere around $w$ and $x$ over a larger 2-sphere surrounding $w$. Summing over cyclic permutations with the appropriate signs, we see that the Jacobi identity follows from associativity, i.e. that the brackets of the mode algebra satisfy the Jacobi identity.
\end{proof}

We now show that if $\CV$ has a $\fg$ Higgs branch flavor symmetry, the $M_a$ play a distinguished role in the Higgs branch chiral ring: they are the components of a moment map for a $\fg$ action on the Higgs branch. For starters, Proposition \ref{prop:chiralOPE} implies the operator $M_a$, being a Higgs branch primary, has regular OPEs with any other Higgs branch primary operator (so long as the BPS bound is satisfied). If $O$ is any other Higgs branch primary operator, then associativity and $N=2$ superconformal symmetry totally constrains the OPE of $M_a$ and the superpartner $\Psi_O$.

\begin{corollary}
	\label{cor:HBsymmprimary}
	Let $\CV$ be an $N=2$ superconformal raviolo vertex algebra satisfying the BPS bound and suppose $\mu_a$ generates a $\fg$ Higgs branch flavor symmetry and let $M_a$ be the Higgs branch primary operators with superpartners $\mu_a$. If $O_i$ are Higgs branch primary operators, with superpartners $\Psi_{O_i}$, transforming in a representation $R$, then the OPE of $M_a$ and $\Psi_{O_i}$ is given by
	\be
	M_a(z) \Psi_{O_i}(w) \sim \Omega^0_{z-w}(-O_j(w)(\rho_a)^j{}_i)
	\ee
	where $\rho_a$ are the representation matrices for the $\fg$ action on $R$
\end{corollary}

Using this expression for the OPE of $M_a$ and the superpartner $\Psi_{O_i}$, we can deduce the following result.

\begin{corollary}
\label{cor:flavorsymmerty}
	Let $\CV$ be an $N=2$ superconformal raviolo vertex algebra satisfying the BPS bound and suppose $\mu_a$ generates a $\fg$ Higgs branch flavor symmetry. Then the Higgs branch $\CM_H[\CV]$ admits a Hamiltonian $\fg$ action with comoment map $\fg \to \CM_H[\CV]$, $T_a \mapsto M_{a}$.
\end{corollary}

There are identical results concerning Coulomb branch flavor symmetries, their proofs are identical to the ones presented above.

\begin{proof}
	This follows from an explicit computation. Let $O_i$ be a collection of Higgs branch primary fields transforming in a representation $R$ of $\fg$, then the Poisson bracket of $M_a$ and $O_i$ is given by
	\be
	\{M_a, O_i\}(z) = \tfrac{1}{2}\oint_z \diff x \bigg( \mu_a(x) O_i(z) - \Psi_{O_i}(x) M_a(z)\bigg) = (\rho_a)^i{}_j O_j(z)
	\ee
	In the second equality we used the OPEs determined by Lemma \ref{lem:HBsymmprimary} and Corollary \ref{cor:HBsymmprimary}.
\end{proof}

\subsubsection{Example: free chiral(s)}
As a first example, consider the $N=2$ superconformal raviolo vertex algebra of a free chiral $SFC$. As we saw above, there are no non-trivial Coulomb branch primary operators, whence $A_C[SFC] = \C$. The Higgs branch chiral ring is generated by the fermion $\eta$, corresponding to the ring of functions on $\CM_H[SFC] = \Pi \C[-1]$. The Poisson bracket of $\eta$ with itself is also easy to compute:
\be
	\{\eta, \eta\}(z) = \oint_z \diff x X(x) \eta(z) = 1
\ee
In particular, the degree $-2$ Poisson structure on the Higgs branch $\Pi\C[-1]$ is non-trivial.

We note that if we were to consider $M$ free chirals then there would be an $SO(M)$ Higgs branch flavor symmetry generated by the bilinears
\be
	 \delta_{[a|c}\norm{X^c \eta_{|b]}} \qquad \text{and} \qquad \norm{\eta_a \eta_b} \; .
\ee

\subsubsection{Example: free hypermultiplet}
As we saw in Section \ref{sec:SCexamples}, the only Coulomb branch primary operator in $FH$ is the identity operator $1$. We also saw that the generating bosons $Z^\alpha$ were Higgs branch primary operators. These operators generate the Higgs branch chiral ring: $A_H[FH] = \C[Z^\alpha]$. We note that $Z^\alpha$ has Higgs branch $R$-charge $1$, so the Higgs branch is identified as $\CM_H[FH] = \C^2[-1]$. The bracket $\{Z^\alpha, Z^\beta\}$ is given by
\be
\{Z^\alpha, Z^\beta\}(z) = \tfrac{1}{2} \oint_z \diff x \bigg(\psi^\alpha(x) Z^\beta(z) - \psi^\beta(x) Z^\alpha(z)\bigg) = \epsilon^{\alpha \beta}
\ee
so we recover the natural degree $-2$ Poisson bracket on $\C^2[-1]$. The bilinears $M^{\alpha \beta} = \tfrac{1}{2}\norm{Z^\alpha Z^\beta}$ are particularly special Higgs branch local operators: their superpartners are the $\fsl(2)$ currents $\epsilon^{(\alpha| \gamma} \norm{\psi_\gamma Z^{|\beta)}}$. We see that these generate an $\fsl(2)$ Higgs branch flavor symmetry. More generally, if we were to consider many copies of this theory there would be an $Sp(2M)$ Higgs branch flavor symmetry generated by the bilinears
\be
	\omega^{(\alpha| \gamma} \norm{\psi_\gamma Z^{|\beta)}} \qquad \text{and} \qquad \tfrac{1}{2}\norm{Z^\alpha Z^\beta}
\ee
where $\omega^{\alpha \beta}$ is the Poisson tensor $\omega_{\alpha \beta}$ corresponding to a choice of symplectic form on $\C^{2M}$.

Putting these results together, we see that the Higgs and Coulomb branches of $FH$ precisely match those of a free hypermultiplet as Poisson varieties: $\CM_C[FH]$ is a point and $\CM_H[FH] = \C^2[-1]$ with Poisson tensor $\epsilon^{\alpha \beta}$.

\subsubsection{Example: free $\CN=4$ vector multiplet}

We saw there were no non-trivial Coulomb branch primary operators in $FV^{\text{pert}}$, so $A_C[FV^{\text{pert}}] = \C$. The Higgs branch chiral ring is a polynomial algebra in two fermionic variables $c, \lambda$. The Poisson bracket of the generators $c, \lambda$ is given by
\be
\{c, \lambda\}(z) = \tfrac{1}{2} \oint_z \diff x \bigg(\phi(x) \lambda(z) + b(x) c(z)\bigg) = 1
\ee
Together with $\{c,c\} = \{\lambda, \lambda\} = 0$, we see that the Higgs branch can be identified with the cotangent bundle $T^*[2](\Pi \C[-1])$ with its natural Poisson structure.

This result is quite different from the full, non-perturbative Higgs and Coulomb branches of a free $\CN=4$ vector multiplet: the non-perturbative Higgs branch is $\Pi \C[-1]$ and the non-perturbative Coulomb branch is $T^*[2]\C^\times$. We can reproduce this answer from $FV$. The only operators that satisfy $j = \frac{q}{2}$ are $1$ and $\lambda$; $\lambda$ is still a Higgs branch primary operator, with superpartner $b$, and Higgs branch $R$-charge $1$. The bracket $\{\lambda, \lambda\}$ vanishes, so we conclude the Higgs branch can be identified as $\CM_H[FV] = \Pi\C[-1]$ with trivial Poisson structure.

The Coulomb branch chiral ring is generated by the monopoles $V_\pm$ together with the boson $\phi$; the monopoles have Coulomb branch $R$-charge $0$, whereas $\phi$ has Coulomb branch $R$-charge $2$. This matches the algebra of functions on $T^*[2]\C^\times$. We find that the Poisson brackets of the generators of the Coulomb branch chiral ring take the expected form
\be
\{\phi, V_\pm\}(z) = \tfrac{1}{2}\oint_z \diff x \bigg(\nu(x) V_\pm(z) \mp \norm{\lambda V_\pm}(x) \phi(z)\bigg) = \pm V_\pm(z)
\ee
Indeed, the superpartner of $\phi$ is the abelian current $\nu$ generating the topological flavor symmetry, hence it is the moment map for the topological flavor symmetry rotating $\C^\times$. If we were to consider $M$ copies of his theory, the $\C^\times$ Coulomb branch symmetry becomes a $(\C^\times)^M$ Coulomb branch flavor symmetry.

\subsection{Superconformal raviolo vertex algebras with superpotential}
\label{sec:superconfsuperpot}

Raviolo vertex algebras can be deformed/coupled by introducing a superpotenial, cf. Section 4.5 of \cite{GarnerWilliams}. For $\CV$ a raviolo vertex algebra, a \emph{superpotential} is a bosonic vector $W \in \CV^2$ of $R$-charge 2 such that $W_{(0)}W$ belongs to the image of the translation operator $\pd$; the pair $(\CV, W)$ is called a raviolo vertex algebra with superpotential. As shown by Lemma 4.5.2 of loc. cit., $D_W: O \to W_{(0)} O$ defines a differential on $\CV$. If $\CV$ is conformal, the superpotential is required to be a primary of spin 1, i.e. the OPE of the stress tensor $\Gamma$ and $W$ is of the form
\be
	\Gamma(z) W(w) \sim \Omega^1_{z-w} W(w) + \Omega^0_{z-w} \pd W(w)\,.
\ee

Now suppose $\CV$ is an $N=2$ superconformal raviolo vertex algebra satisfying the BPS bound and suppose $(\CV, W)$ is a conformal raviolo vertex algebra with superpotential such that $W$ is a primary for the superconformal current $\sigma$ with $q_W = 0$, i.e. the OPE of $W$ and $\sigma$ is regular and hence $D_W \sigma = 0$. The BPS bound implies that the OPE of $W$ and $Q^\pm$ takes the form
\be
	W(z) Q^\pm(w) \sim \Omega^1_{z-w} W^{(1)}_\pm(w) + \Omega^0_{z-w} W^{(0)}_\pm(w)
\ee
where $W^{(0)}_\pm(w) = (W_{(0)}Q^\pm)(w)$ and $W^{(1)}_\pm(w) = (W_{(1)}Q^\pm)(w)$. In order for the $N=2$ superconformal symmetry to be compatible with the differential $D_W$, i.e. $D_W Q^\pm = 0$, we require $W^{(0)}_\pm = 0$; we call such a pair $(\CV, W)$ an $N=2$ \emph{superconformal raviolo vertex algebra with superpotential} (satisfying the BPS bound). Note that $W^{(1)}_\pm$ has $R$-charge $2$, spin $\frac{1}{2}$, and $S$-charge $\pm 1$; we see that they necessarily saturate the BPS bound.

The following result is immediate from the definitions.
\begin{proposition}
	Let $(\CV, W)$ be an $N=2$ superconformal raviolo vertex algebra with superpotential satisfying the BPS bound, then $D_W$ equips $\CV$ with the structure of a dg $N=2$ superconformal raviolo vertex algebra. Moreover, if $O$ is a Higgs or Coulomb branch primary operator then so too is $D_W O$, with superpartner $-D_W \Psi_O$.
\end{proposition}

We see that if $(\CV, W)$ is a $N=2$ superconformal raviolo vertex algebra with superpotential satisfying the BPS bound, then $D_W$ induces a degree 1 differential on the Higgs and Coulomb branch chiral rings. In fact, it is actually a derivation of the full shifted Poisson structure:

\begin{corollary}
	Let $(\CV, W)$ be an $N=2$ superconformal raviolo vertex algebra with superpotential satisfying the BPS bound, and let $O_1, O_2$ be Higgs branch primary operators, then
	\begin{equation*}
		D_W \{O_1, O_2\} = \{D_W O_1, O_2\} + (-1)^{|O_1|}\{O_1, D_W O_2\}
	\end{equation*}
	In particular, $D_W$ equips the Higgs branch chiral ring with the structure of a dg 2-shifted Poisson algebra.
\end{corollary}

There is an identical statement for the Coulomb branch chiral ring.

\begin{proof}
	We compute:
	\be
	\begin{aligned}
		D_W \{O_1, O_2\}(z) & = D_W \bigg(\tfrac{1}{2}\oint_z \diff x \Psi_{O_1}(x) O_2(z) - (-1)^{|O_1||O_2|} \Psi_{O_2}(x)O_1(z) \bigg)\\
		& \hspace{-0.5cm} = \tfrac{1}{2}\oint_z \diff x \bigg(-\big(D_W \Psi_{O_1}(x)\big) O_2(z) + (-1)^{|O_1|}\Psi_{O_1}(x) \big(D_W O_2(z)\big)\\
		& \hspace{-0.5cm} + (-1)^{|O_1||O_2|} \big(D_W \Psi_{O_2}(x)\big)O_1(z) + (-1)^{(|O_1|+1)|O_2|} \Psi_{O_2}(x) \big(D_W O_1(z)\big)\bigg)\\
		& \hspace{-0.5cm} = \{D_W O_1, O_2\}(z) + (-1)^{|O_1|} \{O_1, D_W O_2\}(z)
	\end{aligned}
	\ee
\end{proof}

\subsubsection{Example: perturbative $\CN=4$ gauge theory}
\label{sec:pertgaugetheory}

One important example of an $N=2$ superconformal raviolo vertex algebra with superpotential arises as the supersymmetric analog of the BRST reduction discussed in Section 4.5.2 of \cite{GarnerWilliams}.

Suppose $\CV$ is an $N=2$ superconformal raviolo vertex algebra satisfying the BPS bound and $\mu_a$ generate a $\fg$ Higgs branch flavor symmetry. As before, we denote the accompanying Higgs branch primary operators $M_a$. We then consider the $N=2$ superconformal raviolo vertex algebra $\CV \otimes (FV^{\textrm{pert}}{}^{\otimes \fg})$ and the field
\be
W_{tot} = \tfrac{1}{2} f^a_{bc} \norm{b_a c^b c^c} + f^a_{bc}\norm{\lambda_a c^b \phi^c} - \norm{c^a \mu_a} + \norm{\phi^a M_a}
\ee

\begin{proposition}
	$W_{tot}$ gives $\CV \otimes (FV^{\textrm{pert}}{}^{\otimes \fg})$ the structure of an $N=2$ superconformal raviolo vertex algebra with superpotential.
\end{proposition}

\begin{proof}
	We first note that this construction is an instance of the example considered in loc. cit., so $(\CV, W)$ is indeed a conformal raviolo vertex algebra with superpotential. Moreover, $W$ is a current algebra primary, being a sum of normally-ordered products thereof, and has $S$-charge $0$.
	
	We are then left with checking the OPEs of $W_{tot}$ and $Q^\pm$. A direct computation leads to the following:
	\be
	W_{tot} Q^+ \sim \Omega^1_{z-w} \big(\norm{c^a M_a} + \tfrac{1}{2} f^a_{bc} \norm{\lambda_a c^b c^c}\big) \qquad W_{tot} Q^- \sim 0
	\ee
	from which the assertion follows.
\end{proof}

As shown in Section \ref{sec:SCexamples}, the Higgs branch primary operators of $FV^{\text{pert}}{}^{\otimes \fg}$ are simply normally-ordered products of the fermions $c^a, \lambda_a$; Higgs branch primary operators of $\CV \otimes (FV^{\text{pert}}{}^{\otimes \fg})$ are then normally-ordered products of these fermions and Higgs branch primary operators of $\CV$. The action of $D_{tot} = D_{W_{tot}}$ on the fermionic generators is given by
\be
	D_{tot} c^a = \tfrac{1}{2}f^a_{bc} \norm{c^a c^b} \qquad D_{tot} \lambda_a = f^b_{ca} \norm{c^c \lambda_b} + M_a
\ee
Similarly, the action of $D_W$ on a Higgs branch primary operator $O \in \CV$ is given by
\be
	D_{tot} O(z) = \norm{c^a(z) \bigg(\oint_z \diff x \mu_a(x) O(z)\bigg)} = \norm{c^a \{M_a, O\}}(z)
\ee
Indeed, when restricted to Higgs branch primary operators in $\CV \otimes (FV^{\text{pert}}{}^{\otimes \fg})$ we see that $D_{tot}$ agrees with Poisson bracket with $\norm{c^a M_a} + \tfrac{1}{2}f^a_{bc} \norm{\lambda_a c^b c^c}$.

Putting this together, we conclude the following:
\begin{theorem}
\label{thm:HBpertreduction}
	The Higgs branch chiral ring of $\CV \otimes (FV^{\textrm{pert}}{}^{\otimes \fg})$ together with the differential $D_{tot}$ can be identified as a dg 2-shifted Poisson algebra with the ring of functions on the derived symplectic reduction of $\CM_H[\CV]$ by $\fg$:
	\begin{equation*}
		(A_H[\CV \otimes (FV^{\textrm{pert}}{}^{\otimes \fg})], D_{tot}) = \CO_{\CM_H[\CV]/\!\!\!/\fg}
	\end{equation*}
\end{theorem}

\subsection{Topological Twisting}
\label{sec:toptwisting}
We saw above that an $N=2$ superconformal raviolo vertex algebra $\CV$ satisfying the BPS bound has attached to it two Poisson algebras, the Higgs and Coulomb branch chiral rings $A_H[\CV]$ and $A_C[\CV]$. In this subsection, we show that these chiral rings can often be realized cohomologically.

Heuristically, we aim to recover $A_H[\CV]$ by taking the cohomology with respect to the superconformal generator $Q^+$ as a superpotential deformation. 
We see that this can not be done without a slight modification of the $N=2$ superconformal raviolo vertex algebra: $Q^+$ has $R$-charge 1 and spin $\frac{3}{2}$, whereas a superpotential needs to have $R$-charge $2$ and spin $1$. 
Thankfully, we can use the superconformal current $\sigma$ to remedy this: we choose a new $R$-charge and spin gradings given by Higgs branch $R$-charge $R_B = R+\sigma_{(0)}$ and the Higgs branch spin $J_B = J- \tfrac{1}{2}\sigma_{(0)}$. As mentioned above, the resulting raviolo vertex algebra is still conformal (necessarily of vanishing central charge $0$), but no longer superconformal. Instead, the $N=2$ superconformal algebra reorganizes itself into the \emph{twisted $N=2$ superconformal raviolo vertex algebra}.

The twisted $N=2$ superconformal raviolo vertex algebra contains a stress tensor $\Gamma$ (with vanishing central charge), an abelian raviolo current $\sigma$ (at level $\xi/3$), and two bosonic operators $Q, \wt{Q}$ whose OPEs with $\Gamma$ and $\sigma$ are
\be
\begin{aligned}
	\Gamma (z) Q (w) & \sim \Omega^1_{z-w} Q(w) + \Omega^0_{z-w} \pd Q(w) \qquad &  \sigma (z) Q (w) & \sim \Omega^0_{z-w} Q(w)\\
	\Gamma (z) \wt{Q} (w) & \sim 2\Omega^1_{z-w} \wt{Q}(w) + \Omega^0_{z-w} \pd \wt{Q}(w) \qquad & \sigma(z) \wt{Q}(w) &\sim -\Omega^0_{z-w} \wt{Q}(w)\\
\end{aligned}
\ee
Namely, $Q$ (resp.  $\wt{Q}$) are raviolo Virasoro primaries of spin $1$ (resp. $2$), $R$-charge $2$ (resp. $0$), and $S$-charge $1$ (resp. $-1$). Their OPE with one another is given by
\be
Q (z)\wt{Q}(w) \sim \Omega^2_{z-w} \big(-\xi/3\big) + \Omega^1_{z-w} \big(-\sigma(w)\big) + \Omega^0_{z-w} \big(-\Gamma (w)\big)
\ee
We call a conformal raviolo vertex algebra $\CV$ with operators $\sigma, Q, \wt{Q}$ satisfying the above OPEs a \emph{twisted $N=2$ superconformal raviolo vertex algebra}. Of course, there is a notion of \emph{twisted $N=2$ superconformal raviolo vertex algebra with superpotential}. The following Lemma is immediate from the definitions.

\begin{lemma}
	Let $(\CV, W)$ be a twisted $N=2$ superconformal raviolo vertex algebra with superpotential, then $(\CV, Q+W)$ is a conformal raviolo vertex algebra with superpotential.
\end{lemma}

Now suppose $\CV$ is an $N=2$ superconformal raviolo vertex algebra. It has associated to it two twisted $N=2$ superconformal raviolo vertex algebras. We denote by $\CV_B$ the twisted $N=2$ superconformal raviolo vertex algebra resulting from the above modification, i.e. with stress tensor $\Gamma_B = \Gamma + \tfrac{1}{2} \pd\sigma$, superconformal current $\sigma_B = \sigma$, and bosonic generators $Q_B = Q^+$ and $\wt{Q}_B = Q^-$. Let $D_B$ denote the differential corresponding to $Q_B$; we will call the DG conformal raviolo vertex algebra $(\CV_B, D_B)$ the \emph{$B$-twist} of $\CV$. We can similarly work with the Coulomb branch $R$-charge $R_A = R-\sigma_{(0)}$ and spin $J_A = J + \tfrac{1}{2}\sigma_{(0)}$. This is again twisted $N=2$ superconformal, now with stress tensor $\Gamma_A = \Gamma - \tfrac{1}{2}\pd \sigma$, superconformal current $\sigma_A = - \sigma$, and bosonic generators $Q_A = Q^-$ and $\wt{Q}_A = Q^+$. We denote this twisted $N=2$ superconformal raviolo vertex algebra $\CV_A$ and we call the resulting dg conformal raviolo vertex algebra $(\CV_A, D_A)$ the \emph{$A$-twist} of $\CV$.

If $(\CV, W)$ is an $N=2$ superconformal raviolo vertex algebra with superpotential, then by definition $D_W$ and $D_B$ are commuting differentials on $\CV_B$ and we can view $\CV_B$ as a bicomplex with gradings given by $R$-charge and $S$-charge $(r,q)$; $D_W$ has bidegree $(1,0)$ and $D_B$ has bidegree $(0,1)$. The total degree on this bicomplex is then the Higgs branch $R$-charge and the total differential $D_{B, tot} = D_{B} + D_W$. We call the dg conformal raviolo vertex algebra $(\CV_B, D_{B,tot})$ the \emph{$B$-twist} of $(\CV, W)$.

Our first result about the $B$-twist is the following:

\begin{proposition}
	Let $(\CV_B, D_B)$ be the $B$-twist of an $N=2$ superconformal raviolo vertex algebra, then the stress tensor $\Gamma_B$ is $D_B$-exact.
\end{proposition}

\begin{proof}
	The stress tensor $\Gamma_B$ is the coefficient of $\Omega^0_{z-w}$ in the OPE $Q_B \wt{Q}_B$, hence $D_B \wt{Q}_B = \Gamma_B$.
\end{proof}

An immediate corollary says that non-trivial $D_B$-cohomology classes are concentrated in the sector of vanishing Higgs branch spin.
\begin{corollary}
	\label{cor:Btwistspin}
	Let $(\CV_B, D_B)$ be the $B$-twist of an $N=2$ superconformal raviolo vertex algebra. If $O \in \CV_B$ is a $D_B$-closed operator with Higgs branch spin $j_{B,O} \neq 0$, then $O$ is $D_B$-exact.
\end{corollary}

\begin{proof}
	We use the above proposition to see that $D_B \wt{Q} = \Gamma$; the assertion follows from a direct computation:
	\be
	j_{B,O} O = \Gamma_{B,(1)} O = D_B \big(\wt{Q}_{B,(1)} O\big)
	\ee
\end{proof}

Both of the above results remain true if $(\CV, W)$ is an $N=2$ superconformal raviolo vertex algebra with superpotential and we replace $D_B$ with $D_{B,tot}$.

As $D_B$ is a derivation of the normally-ordered product, the normally-ordered product gives $D_B$-cohomology the structure of a commutative, associative algebra. We now show that there is a natural shifted Poisson bracket directly analogous to the one appearing in the Higgs branch chiral ring. If $O \in \CV_B$ is any operator, we consider a second operator obtained by acting with $\wt{Q}_B$:
\be
O^{[1]}(z) := \oint_z \diff x \wt{Q}_B(x) O(z) = (\wt{Q}_{B,(0)} O)(z)
\ee
called the \emph{holomorphic descendant}, or simply \emph{descendant}, of $O$. The operator $O^{[1]}$ has the opposite (totalized) parity of $O$, Higgs branch $R$-charge $r_B - 1$, and Higgs branch spin $j_B + 1$, where $r_B$ and $j_B$ are the Higgs branch $R$-charge and spin of $O$. If $O$ is a Higgs branch primary operator, we see that $O^{[1]} = \Psi_O$.

It is immediate that $O^{[1]}(z)$ satisfies the \emph{holomorphic descent equation}
\be
D_B O^{[1]} = \pd O - (D_B O)^{[1]}
\ee
which is a direct (holomorphic) analog of Witten's (topological) descent equation \cite{WittenTQFT}. With descendants in hand, we define the following binary operation: 
\be
\{\!\{O_1, O_2\}\!\}(z) := \tfrac{1}{2} \oint_z \diff x O^{[1]}_1(x) O_2(z) - (-1)^{|O_1||O_2|} O^{[1]}_2(x) O_1(z)
\ee
Note that $\{\!\{-,-\}\!\}$ decreases Higgs branch $R$-charge by 2 and preserves Higgs branch spin. Moreover, because $\wt{Q}_B$ has a regular OPE with itself, it follows that $\wt{Q}_{B,(0)}{}^2 = 0$ and hence the descendant of a descendant necessarily vanishes:
\be
O^{[2]}(z) := (\wt{Q}_{B,(0)} O^{[1]})(z) = 0
\ee

\begin{theorem}
\label{thm:twistingHB}
	Let $\CV$ be an $N=2$ superconformal raviolo vertex algebra and let $(\CV_B, D_B)$ be its $B$-twist. The bracket $\{\!\{-,-\}\!\}$ equips $D_B$-cohomology with the structure of a 2-shifted Poisson algebra.
\end{theorem}

If we consider the $B$-twist of an $N=2$ superconformal raviolo vertex algebra with superpotential $(\CV, W)$, we have an analogous descent equation:
\be
	D_{B,tot} O^{[1]} = \pd O - (D_{B,tot} O)^{[1]}
\ee
The following proof is unchanged if we replace $D_B$ by $D_{B,tot}$, i.e. $H^\bullet(\CV_B, D_{B,tot})$ also has the structure of a 2-shifted Poisson algebra.

\begin{proof}
	The necessary skew-symmetry and Higgs branch $R$-charge are immediate from the definition. Much of the proof will be the same as that of Theorem \ref{thm:poissonHB} and the analogous result for $N=2$ superconformal vertex algebras proven in Section 2.3 of \cite{LianZuckerman}, so we will be brief in those portions to avoid repetition.
	
	We first show that $D_B$ is a derivation of the bracket:
	\be
	\begin{aligned}
		D_B \{\!\{O_1, O_2\}\!\}(z) & = \{\!\{D_B O_1, O_2\}\!\}(z) + (-1)^{|O_1|} \{\!\{O_1, D_B O_2\}\!\}(z)\\
		& -\tfrac{1}{2} \oint_z \diff x \bigg(\big(\pd O_1(x)) O_2(z) - (-1)^{|O_1||O_2|} \big(\pd O_2(x)\big) O_1(z)\bigg)\\
		& = \{\!\{D_B O_1, O_2\}\!\}(z) + (-1)^{|O_1|} \{\!\{O_1, D_B O_2\}\!\}(z)
	\end{aligned}
	\ee
	The first equality uses the descent equation, and the second equality follows from the fact that can be no term proportional to $\Omega^0_{x-z}$ in the OPE $\pd O(x) O'(z)$ for any $O, O'$.
	
	To show this bracket satisfies the Jacobi identity, we need the descendant of the bracket $\{\!\{O_1, O_2\}\!\}$:
	\be
	\begin{aligned}
		\{\!\{O_1, O_2\}\!\}^{[1]}(z) & = \tfrac{1}{2}(-1)^{|O_1|} \oint_z \diff x O_1^{[1]}(x) O_2^{[1]}(z)\\
		& + \tfrac{1}{2} (-1)^{|O_1| + (|O_1|+1)(|O_2|+1)} \oint_z \diff x O_2^{[1]}(x) O_1^{[1]}(z)
	\end{aligned}
	\ee
	cf. Eq. \eqref{eq:superpartnerbracket}. With this expression, $\{\!\{\{\!\{O_1, O_2\}\!\}, O_3\}\!\}$ is given by
	\be
	\begin{aligned}
		\{\!\{\{\!\{O_1, O_2\}\!\}, O_3\}\!\}(w) & = \tfrac{1}{4}(-1)^{|O_1|}\oint_w \diff z \oint_z \diff x O_1^{[1]}(x) O_2^{[1]}(z) O_3(w)\\
		& \hspace{-0.5cm} + \tfrac{1}{4}(-1)^{(|O_1|+1)(|O_2|+1)+|O_1|}\oint_w \diff z \oint_z \diff x O_2^{[1]}(x) O_1^{[1]}(z) O_3(w)\\
		& \hspace{-0.5cm} + \tfrac{1}{4}(-1)^{|O_3|(|O_1|+|O_2|+1)}\oint_w \diff z \oint_w \diff x O_3^{[1]}(z) O_1^{[1]}(x) O_2(w)\\
		& \hspace{-0.5cm} - \tfrac{1}{4}(-1)^{|O_3|(|O_1|+|O_2|+1) + |O_1||O_2|}\oint_w \diff z \oint_w \diff x O_3^{[1]}(z) O_2^{[1]}(x) O_1(w)
	\end{aligned}
	\ee
	As in the proof of Theorem \ref{thm:poissonHB}, we see that this vanishes after summing over cyclic permutations (with suitable signs) due to associativity. Note that the Jacobi identity holds exactly, cf. Section 2.3 of \cite{LianZuckerman}.
	
	Finally, to show that $\{\!\{-, O\}\!\}$ is a derivation of the normally-ordered product, we need the descendant of a normally-ordered product:
	\be
	\norm{O_1 O_2}^{[1]} = \norm{O_1^{[1]} O_2} + (-1)^{|O_1|}\norm{O_1 O_2^{[1]}}
	\ee
	cf. Proposition \ref{prop:chiralOPE}. Unlike the Jacobi identity, we find that $\{\!\{-, O\}\!\}$ is only a derivation of the normally-ordered product up to homotopy. A suitable homotopy is given by
	\be
	\begin{aligned}
		n(O_1, O_2, O)(w) & = \tfrac{1}{2}\sum\limits_{l \geq 0} \frac{1}{l+1}\bigg((-1)^{|O_1|}\oint_w \diff z \Omega^l_{z-w} O^{[1]}_1(z) \oint_w \diff x (x-w)^{l+1} O^{[1]}_2(x) O(w)\\
		& \qquad +(-1)^{|O_1||O_2|}\oint_w \diff z \Omega^l_{z-w} O^{[1]}_2(z) \oint_w \diff x (x-w)^{l+1} O^{[1]}_1(x) O(w) \bigg)\\
	\end{aligned}
	\ee
	from which a straight-forward computation leads to
	\be
	\begin{aligned}
		D_B n(O_1, O_2, O) & = \{\!\{\norm{O_1 O_2}, O\}\!\} -\norm{O_1\{\!\{O_2, O\}\!\}} - (-1)^{|O||O_2|}\norm{\{\!\{O_1, O\}\!\}O_2}\\
		& \hspace{-1cm} + n(D_B O_1, O_2, O) + (-1)^{|O_1|}n(O_1, D_B O_2, O) - (-1)^{|O_1|+|O_2|} n(O_1, O_2, D_B O) 
	\end{aligned}
	\ee
	In the same way, $\{\!\{O, -\}\!\}$ is only a derivation of the normally-ordered product up to homotopy; by the skew-symmetry of the bracket, a suitable homotopy is given by $(-1)^{|O|(|O_1|+|O_2|)} n(O_1, O_2, O)$.
\end{proof}

We note that the bracket described in \cite{LianZuckerman} is such that $\{\!\{O, -\}\!\}$ is a derivation without the need of a homotopy, whereas the skew-symmetry of the bracket requires a homotopy. We have chosen a bracket that is skew-symmetric without the need for a homotopy, in exchange for the need of a homotopy in this Leibniz rule.

As we have seen, there are now four 2-shifted Poisson algebras associated to any $N=2$ superconformal raviolo vertex algebra $\CV$ satisfying the BPS bound: the chiral rings $A_H[\CV], A_C[\CV]$ introduced in Section \ref{sec:chiralrings} and the cohomologies $H^\bullet(\CV_B, D_B)$, $H^\bullet(\CV_A, D_A)$. Focusing on the Higgs branch/$B$-twist, we see that the product structure on these rings arises from normally-ordered product on $\CV$ and the brackets are quite similar in form, so it is reasonable to expect that these rings are related to one another. Indeed, we have the following.

\begin{proposition}
\label{prop:naturalmap}
	There is a natural map $A_H[\CV] \to H^\bullet(\CV_B, D_B)$ of 2-shifted Poisson algebras taking a Higgs branch primary operator $O$ to its $D_B$-cohomology class $[O]$. This map is an isomorphism if all operators with vanishing Higgs branch spin $j_B = 0$ are Higgs branch primary operators.
\end{proposition}

\begin{proof}
	Any Higgs branch primary operator is necessarily $D_B$-closed, so the statement makes sense at the level of commutative, associative algebras. To see that this is a shifted Poisson morphism, we note that if $O_1, O_2$ are Higgs branch primary operators then $\{O_1, O_2\} = \{\!\{O_1, O_2\}\!\}$ because $\Psi_O = O^{[1]}$ for such operators.
	
	Now suppose all operators satisfying $j_B = 0$ are Higgs branch primary operators. Corollary \ref{cor:Btwistspin} implies that any $D_B$-closed operator $O$ that does not satisfy $j_B = 0$ is $D_B$-exact. In particular, any non-zero cohomology class can be represented by a Higgs branch primary operator, providing the desired inverse.
\end{proof}

All operators satisfying $j_B = 0$ (resp. $j_A = 0$) were Higgs (resp. Coulomb) branch primary operators for each of the examples in Section \ref{sec:SCexamples} except $FV^\text{pert}$, where the operators $c$ and $\phi$ have $j_A = 0$ but neither is a Coulomb branch primary operator. Applying this proposition to those examples yields the following.

\begin{corollary}
	Let $\CV$ be one of the examples $FH$, $FV$, or $SFC$ in Section \ref{sec:SCexamples}. The Higgs and Coulomb branch chiral rings of $\CV$ are isomorphic to the cohomologies of its $B$- and $A$-twists, respectively.
	\begin{equation*}
		A_H[\CV] \simeq H^\bullet(\CV_B, D_B) \qquad A_C[\CV] \simeq H^\bullet(\CV_A, D_A)
	\end{equation*}
\end{corollary}

Although $FV^\text{pert}$ doesn't satisfy the above condition, we can still show the above map is an isomorphism.

\begin{proposition}
	The Higgs and Coulomb branch chiral rings of $FV^{\textrm{pert}}$ are isomorphic to the cohomologies of its $B$- and $A$-twists, respectively.
	\begin{equation*}
		A_H[FV^{\textrm{pert}}] \simeq H^\bullet(FV^{\textrm{pert}}_B, D_B) \qquad A_C[FV^{\textrm{pert}}] \simeq H^\bullet(FV^{\textrm{pert}}_A, D_A)
	\end{equation*}
\end{proposition}

\begin{proof}
	As all operators in $FV^\text{pert}$ satisfying $j_B = 0$ are Higgs branch primary operators, it suffices to consider the $A$-twist. All non-trivial operators satisfying $j_A = 0$ are realizable as normally-ordered products of $\phi$ and $c$; the action of $D_A$ on these generators is
	\be
		D_A c = \phi \qquad D_A \phi = 0
	\ee
	from which it follows that $D_A$-cohomology is generated by $1$, exactly matching $A_C[FV^\text{pert}]$.
\end{proof}

More generally, it is not clear what the necessary and sufficient conditions are for the maps in Proposition \ref{prop:naturalmap} to be isomorphisms.

\bibliography{ravioloHC}
\bibliographystyle{ieeetr}

\end{document}